\documentclass[letterpaper,leqno]{amsart}
\usepackage[top=1.5in, bottom=1.1in, left=1.2in, right=1.2in]{geometry}
\usepackage{amsmath,amsfonts,amssymb,amsthm}
\usepackage{breqn}
\numberwithin{equation}{section}

\newcommand{\ud}{\,\mathrm{d}}
\newcommand{\supp}{\mathrm{supp}\,}
\newcommand{\adj}{\mathrm{ad}}
\newcommand{\Id}{\mathrm{Id}}
\newcommand{\sym}{\mathrm{Sym}}

\newtheorem{thm}{Theorem}
\newtheorem{lemma}[thm]{Lemma}
\newtheorem{cor}[thm]{Corollary}
\newtheorem{prop}[thm]{Proposition}
\newtheorem{claim}{Claim}
\newtheorem{example}{Example}
\newtheorem{remark}{Remark}

\begin{document}
\title{Multi-center Vector Field Methods for Wave Equations}
\author{Avy Soffer, Jianguo Xiao}
\date{}
\maketitle

\section*{Abstract}
We develop the method of vector-fields to further study Dispersive Wave Equations.

Radial vector fields are used to get a-priori estimates such as the Morawetz estimate on solutions of Dispersive Wave Equations.

A key to such estimates is the repulsiveness or nontrapping conditions on the flow corresponding to the wave equation.
Thus this method is limited to potential perturbations which are repulsive, that is the radial derivative pointing away from the origin.
In this work, we generalize this method to include potentials which are repulsive relative to a line in space (in three or higher dimensions), among other cases.
This method is based on constructing multi-centered vector fields as multipliers, cancellation lemmas
and energy localization.

\section{Introduction and Notation}
 We consider the Schr\"odinger equation in three or higher dimensions.
Most of the analysis is done for the Schr\"odiger equation with a potential term only ($I\equiv 0$):
$$
 i\frac{\partial \psi}{\partial t} = (-\Delta + V(x)) \psi + \lambda I(x,t, |\psi|)\psi.
$$
As applications, our results provide a method for proving decay estimates for a large class of time dependent Hamiltonians, 
as well as nonlinear Dispersive equations.
Previously, such estimates were impossible, since the known proofs are generally based on resolvent techniques,
near threshold energies at least.

A-priori estimates play a fundamental role in controlling the large time behavior of Dispersive Wave Equations.
Besides the classical energy estimates, a key class of estimates are the Morawetz type bounds.
The Morawetz estimates were first introduced by Cathleen Morawetz \cite{morawetz1968time,morawetz1977decay} for nonlinear Klein-Gordon equation.
Later Lin and Strauss \cite{lin1978decay} introduced the Morawetz estimates into the context of NLS equation to prove the
scattering of defocusing NLS equation.
Such estimates can be obtained by constructing a multiplier $\gamma$ which has positive commutator with the Hamiltonian $H$, 
that is, for some operator $B$,
\begin{equation} \label{eqn:commutator-estimate}
 i[H, \gamma] \geq B^* B \geq 0,
\end{equation}
in the sense of forms. For Schr\"odinger equations, using Ehrenfest Theorem, one get a monotonic formula:
\begin{equation}
 \frac{\ud}{\ud t} (\psi, \gamma \psi) = (\psi, i[H, \gamma] \psi) \geq (\psi, B^* B \psi) = \| B\psi \|_{L^2}^2.
\end{equation}
Here $(\cdot, \cdot)$ is the inner product on $L^2_{x}(\mathbb{R}^n)$.
Then by integration over time and conservation laws, it follows Morawetz type estimate
\begin{equation}
\int \| B \psi(t) \|^2 \ud t \leq 2 \sup_{t \in \mathbb{R}}\|\gamma \psi(t) \|_{L^2} \| \psi(t)\|_{L^2}.
\end{equation}
The commutator estimate (\ref{eqn:commutator-estimate}) also implies that Morawetz type bounds hold for the Wave Equation.
To see that, one can use the Heisenberg type identity from the paper \cite{blue2005wave}:
\begin{equation}
 \frac{\ud}{\ud t}((u, A u_t) - (u_t, A u)) = (u, [H,A]u).
\end{equation}
Here $u$ is the solution to Wave Equation $\partial_t^2 u + Hu =0$, and $A$ is time independent operator.

These multipliers $\gamma$ are usually generated by radial vector fields, i.e. vector fields centered at the origin.
In the original works of Morawetz, she introduced and used the radial vector fields
$\vec{f} = x = (x_1, \dots, x_n)$ and $\vec{f} = x/|x| = (x_1/|x|, \dots, x_n/|x|)$ for $n \geq 3$.
The corresponding multipliers $M_f = \vec{f} \cdot (-i\nabla_x) + (-i\nabla_x) \cdot \vec{f}$ then play a fundamental role
in establishing global existence and scattering theory for Schr\"odinger type equations, as well as other wave equations.
Because the commutators with the free Hamiltonian are positive, i.e. $i[-\Delta, M_f] \geq B^* B$ for some operator $B$.

Further generalization of these vector fields were introduced by many authors,
in different context \cite{lavine1971commutators,sigal1987n,graf1990asymptotic,hunziker2000quantum,Blue-Sof,blue2006uniform,blue2007space,blue2009phase,tataru2008parametrices,dafermos2005proof,soffer2011monotonic,dafermos2010new,luk2012vector}.
By considering such multipliers on product fields, Tao \cite{tao2006nonlinear} proved new kind of a-priori estimates for the Schr\"odinger equation.
A key restriction on the interaction $I(x,t,\psi)$ is repulsiveness:
it is required that
\begin{equation}
 \int_{\mathbb{R}^n} \bar{\psi} \left \{ i[\lambda I(x,t,\psi), M_f] -x \cdot \nabla_x V(x)\right \}  \psi \ud^n x\geq 0.
\end{equation}
Therefore, we are restricted to repulsive (defocusing) nonlinearities, and potentials which are repulsive(see e.g.\cite{metcalfe2009decay,killip2015energy}) :
\begin{equation}
 -x \cdot \nabla_x V(x) \geq 0.
\end{equation}
So, in particular, the problem of global existence and scattering theory for the NLS equation with a general (even smooth) potential $V(x)$
which is positive, is open for large data. That is for the equation
\begin{equation}
 i\frac{\partial \psi}{\partial t} = (-\Delta + V(x)) \psi + \lambda |\psi|^{p-1} \psi
\end{equation}
with $V(x)\geq 0$, $\lambda >0$, $1+p \geq p_c$, $p_c = \frac{2n}{n-2}$, $n$ is the dimension of space.
This problem is also open for low power nonlinearities.
In both cases, one needs the Morawetz estimate.
The source of the problem with non-repulsive interactions is the existence of bounded (in space) geodesics, for the classical flow.
It is then clear that we can not have a growing quantity along such geodesics,
and it is this growth which is responsible for the positivity in the Morawetz type inequalities.
It is the basis behind the method of vector fields, being a generalization of the idea of Lyapunov function.
For our approach to work then, we need to employ ``Quantum Effects'' as well.
What we show is that we can construct a monotonic quantity under the flow, outside an arbitrary small (in measure) set,
containing the bounded geodesics.
Then we use compactness arguments, energy localization and the positivity, via Hardy's inequalities to absorb the negative part.

In this work, we introduce a construction of multi-centered vector fields, 
which we then use to obtain Morawetz type estimates (positive commutators)
for potentials which are repulsive relative to a line rather than a point:
let $x=(x_1, \vec{y})$, $x_1 \in \mathbb{R}$, $\vec{y} \in \mathbb{R}^{n-1}$, $n\geq 3$.
Then our condition of repulsiveness on $V(x) = V(x_1, \vec{y})$ reads
\begin{equation}
 -\vec{y} \cdot \nabla_{\vec{y}} V(x_1, \vec{y}) \geq 0,
\end{equation}
as well as other regularity conditions, but no sign assumptions.
See Theorems \ref{thm:two-bump}, \ref{thm:latice}, \ref{thm:general} and \ref{thm:interaction-morawetz}.

Further generalizations include the proof of Morawetz type estimates for localized frequencies near zero or infinity
- for rather general classes of potentials, not necessarily repulsive.
See Theorems \ref{thm:high-energy} and \ref{thm:low-energy}.

The construction of the multi-centered multipliers involves the following steps:

First, we introduce a cancellation lemma: it states that if a potential bump is repulsive w.r.t. the origin in $\mathbb{R}^n$,
then the sum of multipliers centered at $c$ and $-c$, $c \in \mathbb{R}^n$,
\begin{equation}
 \gamma_c + \gamma_{-c}
\end{equation}
has positive commutator with $V(x)$:
\begin{equation}
 i[V(x), \gamma_c + \gamma_{-c}] \geq \theta(x) \geq 0,
\end{equation}
where
\begin{equation}
\begin{split}
 \gamma_c \psi &\equiv -i \{ (\nabla_x F(|x-c|)) \cdot \nabla_x \psi + \nabla_x \cdot [(\nabla_x F(|x-c|)) \psi] \} \\
 &= -i \sum_j (\partial_j F(|x-c|))(\partial_j \psi) + \partial_j (\partial_j F(|x-c|) \psi),
\end{split}
\end{equation}
$F(|x|)$ is a properly chosen radial function with bounded derivative.
This lemma and a generalization to a sum of $\gamma_c$'s centered on a line plays a key role in the analysis.
Then, the next observation is that for a potential which is repulsive in directions orthogonal to the line connecting $\vec{c}$ and $-\vec{c}$,
one can show, for any $\epsilon >0$, that for $N$ large enough
\begin{equation}
 i[V(x), \sum_{j=1}^{N} (\gamma_{c_j} + \gamma_{-c_j})] \geq \theta(x)
\end{equation}
with $c_j$ all on the same line, and such that $\theta(x) >0$ for all $x = (x_1,y)$ with $|y| > \epsilon$.
Here $x_1$ is the coordinate along the line containing all the $c_i$'s.

Next, one uses frequency decomposition.
On the region $|y| \leq \epsilon$ and low frequency, we use compactness to prove that this contribution vanishes in norm as $\epsilon \to 0$,
and therefore is dominated by the positive operator $i[-\Delta, \sum_{j=1}^{N} (\gamma_{c_j} + \gamma_{-c_j})]$.
For the high frequency part, we use that all regions with negative commutator, are dominated by the Laplacian part of the commutator,
provided the frequency cutoff is large enough, depending only on the size of $|\partial V/\partial x|$.

Other cases are also included, including potentials with nondefinite sign.

Another class of potentials are time dependent potentials.
The simplest cases are potentials which are axially repulsive in our sense and which are also moving,
in a compact interval, along the axial direction.
In particular, if $V(x_1, y)$ is a potential that satisfies our axial repulsiveness conditions,
then similar decay and a-priori estimates hold for the time dependent potential $V(x_1 + \beta(t),y)$, with $\sup_t |\beta(t)| < \beta_0 < \infty$.

Finally, it should be noted that small deviations from the axial axis are allowed:
suppose $V_{j}(x)$ is radial, smooth compactly supported and repulsive: $-x \cdot \nabla_x V_{j}(x) \geq 0$.
Then potentials $V(x)$ of the form
\begin{equation}
 \sum_{j=1}^{N} \beta_j V_{j}(x + a_j), ~~\beta_j >0
\end{equation}
with $a_j = (x_1(j), \vec{y}(j))$, $|\vec{y}(j)| <\delta \ll 1$, $x_1(j) \in \mathbb{R}$, $\vec{y}(j) \in \mathbb{R}^{n-1}$ will satisfy our conditions.

Now let us introduce some notations and preliminary results that we will use later.

Suppose $H = -\Delta + V(x)$, with $V(x)$ smooth and such that $H$ is a selfadjoint operator with $D(H) = D(-\Delta)$.
And we assume the dimension of space is three or higher.

Let $a$, $\sigma$ be some positive numbers, and for $x \in \mathbb{R}^n$, $r = |x|$.
Define
\begin{equation}\langle x \rangle_a = (1+a|x|^2)^{1/2},\end{equation}
\begin{equation}g(r) = \frac{1}{\langle x \rangle_a^{\sigma}} = \frac{1}{(1+ar^2)^{\sigma/2}},\end{equation}
\begin{equation}f(r) = \int_0^r g^2(s) \ud s = \int_0^{r} \frac{1}{(1+as^2)^{\sigma}} \ud s,\end{equation}
\begin{equation}
 M_{\sigma} = \int_0^{\infty} \frac{1}{(1+s^2)^{\sigma}} \ud s, \text{~~then~~} f(\infty) = \lim_{r \to \infty} f(r) = \frac{M_{\sigma}}{\sqrt{a}},
\end{equation}
\begin{equation} F(x) = F(r) = \int_0^{r} f(t) \ud t.\end{equation}
We require $\sigma > 1/2$, so that $M_{\sigma}$ exists and is finite.
We fix $a>0$, and omit the subscript of $\langle x \rangle_a$ in the following context.

Write $F_c(x) \triangleq F(|x-c|)$, where $c \in \mathbb{R}^n$ is the position of the center.
We define the multiplier $\gamma_c$ centered at $c$ as:
\begin{equation}
\gamma_c \triangleq i[-\Delta, F_c(x)] 
= -i(\frac{\partial}{\partial x} \cdot \nabla F_c + \nabla F_c \cdot \frac{\partial}{\partial x}).
\end{equation}
If one choose $f \equiv 1$ instead, then $\gamma_c$ become the multiplier used in proving (interaction) Morawetz estimate:
\begin{equation}
 \gamma_{c}^{Mor}= -i (\frac{x-c}{|x-c|} \cdot \nabla_x + \nabla_x \cdot \frac{x-c}{|x-c|}).
\end{equation}

By direct computation, we have
\begin{equation}\tag{C1}
i[-\Delta, \gamma_c] = -4 \partial_j (F_c)_{jk} \partial_k - \Delta^2 F_c, \end{equation}
\begin{equation}\tag{C2}
i[V(x), \gamma_c] = -2 \nabla F_c \cdot \nabla V(x), \end{equation}
\begin{equation}\tag{C3}
i[H, \gamma_c] = -4 \partial_j (F_c)_{jk} \partial_k - \Delta^2 F_c -2 \nabla F_c \cdot \nabla V(x). \end{equation}

For the function $F(x)$, we have its Hessian matrix:
\begin{equation}\tag{C4}
F_{jk} = \frac{f(r)}{r} \delta_{jk} + \left( -\frac{f(r)}{r} + g^2(r)\right) \frac{x_j x_k}{r^2}.\end{equation}
Since $F_c(x)$ is a translation of $F(x)$, we have the Hessian matrix of $F_c(x)$:
\begin{equation}\tag{C4'}
(F_c)_{jk} = \frac{f(|x-c|)}{|x-c|} \delta_{jk} + \left( -\frac{f(|x-c|)}{|x-c|} + g^2(|x-c|)\right) \frac{(x_j - c_j) (x_k - c_k)}{|x-c|^2}.
\end{equation}

Notice that the matrix $(x_j x_k)$ is of rank one, so the eigenvalues of $(F_{jk})$ are:
\begin{equation}\tag{C5}
\lambda_1 = g^2(r), \lambda_2 = \dots = \lambda_n = \frac{f(r)}{r}.\end{equation}
And the corresponding eigenvectors are:
$v_1 = (x_1,x_2,\dots,x_n)^T$, and $v_2, \dots, v_n$ are any $n-1$ independent vectors that are orthogonal to $v_1$.
The lowest eigenvalue is $\lambda_1 = g^2(r)$, thus we have the following:
\begin{equation}\tag{C6}
\Delta F = (n-1) \frac{f(r)}{r} + g^2(r), \text{~and~} -4 \partial_j F_{jk} \partial_k \geq -4 \partial_j g^2(r) \partial_j.	\end{equation}
We then compute:
\begin{equation}\tag{C7}
-\Delta^2 F= \frac{(n-1)(n-3)}{r^2}(\frac{f(r)}{r} - \frac{1}{(1+ar^2)^{\sigma}})
+ \frac{a\sigma (4n-2)}{(1+ar^2)^{\sigma + 1}} + \frac{-4a^2 \sigma (\sigma+1) r^2}{(1+ar^2)^{\sigma + 2}}.
\end{equation}

We also compute the derivatives of $g(r)$, which we will use later:
\begin{equation}\label{C8}
\frac{\partial}{\partial r} g(r) = \frac{-a\sigma r}{(1+ar^2)^{\sigma/2+1}},
\end{equation}
\begin{equation}\label{C9}
\frac{\partial^2}{\partial r^2} g(r) = \frac{-a\sigma}{(1+ar^2)^{\sigma/2+1}} + \frac{a^2\sigma(\sigma+2)r^2}{(1+ar^2)^{\sigma/2+2}},
\end{equation}
\begin{equation}\label{C10}
\Delta g(r) = \frac{-na\sigma}{(1+ar^2)^{\sigma/2+1}} + \frac{a^2\sigma(\sigma+2)r^2}{(1+ar^2)^{\sigma/2+2}}.
\end{equation}

Now we are ready to estimate the commutator $i[-\Delta, \gamma_c]$:
\begin{prop} \label{prop:Delta}
For $n\geq 3$, fixed positive numbers $a$, $\sigma$, and the mutiplier $\gamma_c$ as defined above,
\begin{equation}
\begin{split}
i[-\Delta,\gamma_c] \geq & - (4-\frac{\sigma}{(n-2)^2}) \frac{1}{\langle x-c \rangle^{\sigma}} \Delta \frac{1}{\langle x-c \rangle^{\sigma}} \\
& -4 \partial_j \{\frac{f(|x-c|)}{|x-c|} - g^2(|x-c|)\} \{\delta_{jk} - \frac{(x_j-c_j) (x_k-c_k)}{|x-c|^2}\} \partial_k \\
& + \frac{(n-1)(n-3)}{|x-c|^2}(\frac{f(|x-c|)}{|x-c|} - g^2(|x-c|)) \\
& + \frac{\sigma (1-3a|x-c|^2)^2}{4|x-c|^2 (1+a|x-c|^2)^{\sigma+2}} \\
\geq & - (4-\frac{\sigma}{(n-2)^2}) \frac{1}{\langle x-c \rangle^{\sigma}} \Delta \frac{1}{\langle x-c \rangle^{\sigma}}.
\end{split}\end{equation}
In particular, if $0\leq \sigma < 4(n-2)^2$, then $i[-\Delta, \gamma_c]$ is a positive operator.
\end{prop}

\begin{proof}
From (C1) and (C4),
\begin{equation}
\begin{split}
 i[-\Delta,\gamma_c] = & -4 \partial_j \frac{1}{\langle x-c \rangle^{2\sigma}} \partial_j - \Delta^2 F_c \\
 & -4 \partial_j \{\frac{f(|x-c|)}{|x-c|} - g^2(|x-c|)\} \{\delta_{jk} - \frac{(x_j-c_j) (x_k-c_k)}{|x - c|^2}\} \partial_k.
\end{split}\end{equation}
The third term is positive because of the following claim (with substitution $x \to x-c$):
\begin{claim} \label{claim:1}
$f(r)/r - g^2(r) \geq 0$ and the matrix $(\delta_{jk} - \frac{x_j x_k}{|x|^2})$ is positive semi-definite.
\end{claim}
\begin{proof}[Proof of Claim \ref{claim:1}]
$f(r) = \int_0^r g^2(s) \ud s \geq r \min_{s \in [0,r]}(g^2(s)) = r g^2(r)$, since $g^2(s)$ is decreasing.
Thus we get $f(r)/r -g^2(r) \geq 0$.

The matrix $(\frac{x_j x_k}{|x|^2}) = (\frac{x}{|x|}) (\frac{x}{|x|})^T$ is symmetric and of rank one, and its only nonzero eigenvalue is 1.
So the matrix $(\delta_{jk} - \frac{x_j x_k}{|x|^2})$ is still symmetric, with eigenvalues: $\lambda_1 = 0, \lambda_2 = \dots = \lambda_n = 1$.
And hence the matrix is positive semi-definite.
\end{proof}

Compute using (C7) and (\ref{C10})
\begin{equation}
\begin{split}
 & -4 \partial_j \frac{1}{\langle x-c \rangle^{2\sigma}} \partial_j - \Delta^2 F_c\\
= & -4 \frac{1}{\langle x-c \rangle^{\sigma}} \Delta \frac{1}{\langle x-c \rangle^{\sigma}} + 4g(|x-c|)(\Delta g(|x-c|)) - \Delta^2 F_c \\
= & -4 \frac{1}{\langle x-c \rangle^{\sigma}} \Delta \frac{1}{\langle x-c \rangle^{\sigma}} + \frac{-2a\sigma +2a^2 \sigma |x-c|^2}{(1+a|x-c|^2)^{\sigma+2}} \\
& + \frac{(n-1)(n-3)}{|x-c|^2}(\frac{f(|x-c|)}{|x-c|} - g^2(|x-c|))
\end{split}
\end{equation}

Notice that $\frac{f(r)}{r} \geq \frac{1}{(1+ar^2)^{\sigma}}$, and that $n \geq 3$, so we have
\begin{equation}\frac{(n-1)(n-3)}{|x-c|^2}(\frac{f(|x-c|)}{|x-c|} - g^2(|x-c|)) \geq 0.\end{equation}
If we use Hardy's inequality, we get
\begin{equation}\begin{split}
& - \frac{\sigma}{(n-2)^2} \frac{1}{\langle x-c \rangle^{\sigma}} \Delta \frac{1}{\langle x-c \rangle^{\sigma}} + \frac{-2a\sigma +2a^2 \sigma |x-c|^2}{(1+a|x-c|^2)^{\sigma+2}} \\
\geq & \frac{\sigma}{(n-2)^2} \frac{1}{\langle x-c \rangle^{\sigma}} \frac{(n-2)^2}{4|x-c|^2} \frac{1}{\langle x-c \rangle^{\sigma}} + \frac{-2a\sigma +2a^2 \sigma |x-c|^2}{(1+a|x-c|^2)^{\sigma+2}} \\
= & \frac{\sigma(1+a|x-c|^2)^{2} + 4|x-c|^2 (-2a\sigma +2a^2 \sigma |x-c|^2)}{4|x-c|^2 (1+a|x-c|^2)^{\sigma+2}} \\
= & \frac{\sigma (1-3a|x-c|^2)^2}{4|x-c|^2 (1+a|x-c|^2)^{\sigma+2}} \geq 0.
\end{split}\end{equation}
Sum up the above inequalities, we obtain the desired result.
\end{proof}

\section{Two-Bump Potential}
In this section, we consider the easiest case of nonrepulsive potential: $V(x)$ consists of two spherically symmetric bump functions.
That is $H = -\Delta + V_{-1}(|x+b|) + V_1(|x-b|)$, 
$V_{-1}$ and $V_1$ are real valued radially decreasing smooth potentials of compact support, 
with centers at $x=-b$ and $x=b$ respectively.
Under these conditions, we have that $H$ is a selfadjoint operator, $D(H) = D(-\Delta)$.

The main purpose of this section is to construct $\gamma_N$ as a sum of $\gamma_c$'s, such that $i[H, \gamma_N]$ be a positive operator.

\begin{thm} \label{thm:two-bump}
Let $H$ be a Schr\"odinger operator with two-bump potential, $H = -\Delta + V_{-1}(|x+b|) + V_1(|x-b|)$.
Assume $V_{-1}(|x+b|)$ and $V_1(|x-b|)$ are real valued radially decreasing smooth potentials of compact support, with centers at $x=-b$ and $x=b$ respectively.
If $1/2 < \sigma < 4(n-2)^2$, 
then for any $0<\epsilon<1$, there exists $N_{\epsilon}$, such that for any $N \geq N_{\epsilon}$,
\begin{equation}
\begin{split}
i[H,\gamma_N] &\geq (1 -\epsilon) i[-\Delta, \gamma_N] \\
&\geq -(1- \epsilon)(4 - \frac{\sigma}{(n-2)^2}) \sum_{k=-N}^N \frac{1}{\langle x- kb \rangle^{\sigma}} \Delta \frac{1}{\langle x- kb \rangle^{\sigma}}.
\end{split}
\end{equation}
Here $\gamma_N = \sum_{k=-N}^N \gamma_{c_k}$, with $c_k = kb$.
\end{thm}

To prove Theorem \ref{thm:two-bump}, we need the following lemma to control the size of $i[V_{-1}+V_1,\gamma_N]$.

\begin{lemma}[Cancellation lemma]
For any radially symmetric and decreasing real valued $C^1$ potential $V_0(x) = V_0(|x|)$ in $\mathbb{R}^n$,
and any $c = (c_1, \vec{0}) \in \mathbb{R}^n$ with $\vec{0} = (0, \dots, 0) \in \mathbb{R}^{n-1}$, we have
\begin{equation}i[V_0(x), \gamma_{-c} + \gamma_{c}] \geq -2 V_0'(|x|) \min\{f(|x+c|),f{|x-c|}\} \frac{2|\vec{y}|^2}{|x|} \frac{1}{|x|+|c|} \geq 0.
\end{equation}
Here we write $x = (x_1, \vec{y}) \in \mathbb{R}^n$.
\end{lemma}

\begin{proof}
Without loss of generality, we can assume $c_1 \geq 0$. First we assume $x_1 \geq 0$, then similar result will follow for $x_1 \leq 0$.
\begin{equation}i[V_0(x), \gamma_{-c} + \gamma_{c}] = -2 V_0'(r) (f(r_1) \frac{x + c}{r_1} \cdot \frac{x}{r} + f(r_2) \frac{x - c}{r_2} \cdot \frac{x}{r})\end{equation}
Here $r = |x|$, $r_1 = |x + c|$ and $r_2 = |x - c|$. Since $V_0'(r) \leq 0$, we only have to prove $f(r_1) \frac{x + c}{r_1} \cdot \frac{x}{r} + f(r_2) \frac{x - c}{r_2} \cdot \frac{x}{r} \geq 0$. By assumption, $c_1 >0$ and $x_1 \geq 0$, we have
\begin{equation}\frac{x + c}{r_1} \cdot \frac{x}{r} = \frac{(x_1 + c_1) x_1 + |\vec{y}|^2}{r_1 r} \geq 0\end{equation}
\begin{equation} \label{eqn:monotone}
r_1 = \sqrt{(x_1+c_1)^2 +|\vec{y}|^2} \geq \sqrt{(x_1-c_1)^2 +|\vec{y}|^2} = r_2
\end{equation}
The later inequality (\ref{eqn:monotone}) implies that $f(r_1) \geq f(r_2) \geq 0$ by the monotonicity of function $f$. So we have
\begin{equation}i[V_0(x), \gamma_{-c} + \gamma_{c}] \geq -2 V_0'(r) f(r_2) (\frac{x + c}{r_1} \cdot \frac{x}{r} + \frac{x - c}{r_2} \cdot \frac{x}{r})\end{equation}
In order to prove that $i[V_0(x), \gamma_{-c} + \gamma_{c}] \geq 0$, we only need to prove
$\frac{x + c}{r_1} \cdot \frac{x}{r} + \frac{x - c}{r_2} \cdot \frac{x}{r} \geq 0$.

\begin{equation}\begin{split}
& \frac{x + c}{r_1} \cdot \frac{x}{r} + \frac{x - c}{r_2} \cdot \frac{x}{r} \\
=~& \frac{x_1^2 + |\vec{y}|^2 + c_1 x_1}{r r_1} + \frac{x_1^2 + |\vec{y}|^2 - c_1 x_1}{r r_2}\\
=~& \frac{|\vec{y}|^2}{r} (\frac{1}{r_1} + \frac{1}{r_2}) + \frac{x_1^2 (r_1 + r_2)}{r r_1 r_2} + \frac{c_1 x_1 (r_2 - r_1)}{r r_1 r_2} \\
=~& \frac{|\vec{y}|^2}{r} (\frac{1}{r_1} + \frac{1}{r_2}) + \frac{x_1^2 (r_1 + r_2)}{r r_1 r_2} + \frac{-4 c_1^2 x_1^2}{r r_1 r_2 (r_1 + r_2)} \\
=~& \frac{|\vec{y}|^2}{r} (\frac{1}{r_1} + \frac{1}{r_2}) + \frac{x_1^2 (r_1 + r_2)^2 - x_1^2 (r_1^2 + r_2^2 -2r_1 r_2 \cos(\theta))}{r r_1 r_2 (r_1 + r_2)} \\
=~& \frac{|\vec{y}|^2}{r} (\frac{1}{r_1} + \frac{1}{r_2}) + \frac{2 x_1^2 (1 + \cos (\theta))}{r (r_1 + r_2)} \\
\geq ~& \frac{|\vec{y}|^2}{r} (\frac{1}{r_1} + \frac{1}{r_2}) \geq \frac{2|\vec{y}|^2}{r} \frac{1}{r+|c|} \geq 0
\end{split}\end{equation}
Here we used:
\begin{equation}(r_2 - r_1)(r_1 + r_2) = r_2^2 - r_1^2 = (x_1 - c_1)^2 - (x_1 + c_1)^2 = -4 c_1 x_1;\end{equation}
and the cosine law:
\begin{equation}4 c_1^2 = r_1^2 + r_2^2 -2 r_1 r_2 \cos(\theta),\end{equation}
where $\theta$ is the angle between $x+c$ and $x-c$.

From the above computation, we get
\begin{equation}\begin{split}
& i[V_0(x), \gamma_{-c} + \gamma_{c}] \\
\geq & -2 V_0'(r) f(r_2) (\frac{x + c}{r_1} \cdot \frac{x}{r} + \frac{x - c}{r_2} \cdot \frac{x}{r}) \\
\geq & -2 V_0'(r) f(r_2) \frac{2|\vec{y}|^2}{r} \frac{1}{r+|c|} \geq 0
\end{split}\end{equation}

When $x_1 \leq 0$, with the same computation, we have
\begin{equation}\begin{split}
& i[V_0(x), \gamma_{-c} + \gamma_{c}] \\
\geq & -2 V_0'(r) f(r_1) \frac{2|\vec{y}|^2}{r} \frac{1}{r+|c|} \geq 0
\end{split}\end{equation}
Thus we proved the lemma.
\end{proof}

As one can see from the cancellation lemma,
\begin{equation}
 i[V_0(x), \gamma_{-c} + \gamma_{c}] \gtrsim |V_0'(r)| f(\infty) \frac{|\vec{y}|^2}{r} \frac{1}{r+|c|}
\end{equation}
as $|c| \to \infty$. If we take $\gamma_N = \sum_{k=-N}^{N} \gamma_{c_k}$, where $c_k = kb$,
then the gain from the cancellation lemma could be very large, as $\sum_{k=1}^{N} 1/(r+c_k) \sim \log (Nb)$.

\begin{cor} \label{cor:log-accumulate}
Assume $V_0(x) = V_0(|x|)$ is a real valued, radially decreasing $C^1$ function in $\mathbb{R}^n$, and $\Omega \subset B_R(0) \subset \mathbb{R}^n$ is a compact set.
Let $b \neq 0 \in \mathbb{R}^n $, then for any $\delta > 0$, there is a uniform estimate for
$\{x = (x_1, \vec{y}) \in \Omega: |\vec{y}| > \delta\}$:
\begin{equation}i[V_0(x), \gamma_N] \gtrsim |V_0'(|x|)| f(\infty) \frac{\delta^2}{R} \log(N|b|/R).\end{equation}
Especially, if $V_0(x)$ is compactly supported, we can take $\Omega = \supp V_0(x)$.
\end{cor}

\begin{proof}
Pair up the symmetric $\gamma_c$'s, and use the result of cancellation lemma, then for some $M$ depends on $R$ and $f$,
\begin{equation}
 \begin{split}
 i[V_0(x), \gamma_N] \geq & -2 V_0'(|x|) \sum_{k=1}^{N} \min\{f(|x+c_k|),f(|x-c_k|)\} \frac{2|\vec{y}|^2}{|x|} \frac{1}{|x|+|c_k|} \\
 \gtrsim & |V_0'(|x|)| f(\infty) \frac{\delta^2}{R} \sum_{k=M}^{N} \frac{1}{R+k|b|} \\
 \sim & |V_0'(|x|)| f(\infty) \frac{\delta^2}{R} \log(N|b|/R), \text{~as~} N \to \infty
 \end{split}
\end{equation}
\end{proof}

If the potential function has only one bump, then of course one can choose $\gamma_c$'s symmetric w.r.t. the origin.
However, for $H = -\Delta + V_{-1}(|x+b|) + V_1(|x-b|)$, it is impossible to choose $\gamma_c$'s symmetric w.r.t. both $x=-b$ and $x=b$ at the same time.
So one should choose $\gamma_c$'s as symmetric as possible; our choice here is $\gamma_N = \sum_{k=-N}^{N} \gamma_{c_k}$, where $c_k = kb$.
By the cancellation lemma, the only negative terms of $i[H,\gamma_N]$ come from $i[V_{-1}(|x+b|), \gamma_{c_{N-1}} + \gamma_{c_N}]$ and $i[V_1(|x-b|), \gamma_{c_{-(N-1)}} + \gamma_{c_{-N}}]$,
after combining symmetric terms.

Notice that
\begin{equation}|i[V_{-1}(|x+b|), \gamma_{c_{N-1}} + \gamma_{c_N}]| \leq 4 f(\infty) V_{-1}'(|x+b|)|\end{equation}
for any number $N$, and similar bound holds for $|i[V_1(|x-b|), \gamma_{c_{-(N-1)}}  + \gamma_{c_{-N}}]|$.
This means that the negative terms do not grow as $N$ increases, and that they are bounded by fixed functions.
Then Corollary \ref{cor:log-accumulate} enables us to shrink the support of negative part of $i[V_{-1} + V_1, \gamma_N]$ to a tube of radius $\delta$ by taking $N$ large enough.

So we have the following estimate of $i[V_j, \gamma_N]$:
\begin{prop}\label{prop:tube}
Assume $V_{-1}(|x+b|)$ and $V_1(|x-b|)$ are real valued radially decreasing smooth potentials of compact support, with centers at $x=-b$ and $x=b$ respectively.
For any integer $N$, and $j=-1, 1$, let $S^{(j)}_N = \{x: i[V_j(x),\gamma_N] < 0\}$, and $\chi^{(j)}_N (x) = \chi_{S^{(j)}_N} (x)$ be the characteristic function of $S^{(j)}_N$. Then
\begin{equation}i[V_j(|x-jb|), \gamma_N] \geq -4 f(\infty) |V_{j}'(|x - jb|)| \chi^{(j)}_N (x)\end{equation}
and $S^{(j)}_N \subset \{x=(x_1,\vec{y}) \in \mathbb{R}^n: |\vec{y}| \leq \delta_N\} \cap \supp V_j$ with $\delta_N \to 0$ as $N$ goes to infinity.
\end{prop}

\begin{proof}
We only prove it for $V_{-1}$. Use Corollary (\ref{cor:log-accumulate}), for any $\delta>0$, if $x\notin \mathbb{R} \times B_{\delta}(0)$ we have
\begin{equation}
\begin{split}
 i[V_{-1}(|x+b|), \gamma_N] =& i[V_{-1}(|x+b|), \sum_{k=-(N-1)}^{N-1} \gamma_{c_{k-1}}] + i[V_{-1}(|x+b|), \gamma_{c_{N-1}} + \gamma_{c_N}] \\
 \gtrsim & |V_{-1}'(|x+b|)| f(\infty) \frac{\delta^2}{R} \log(Nb/R) - 4 f(\infty) V_{-1}'(|x+b|)| \\
 \geq & 0, \text{~~for $N \gtrsim \exp(1/\delta^2)$ large enough.}
\end{split}
\end{equation}
So the region where $i[V_{-1}(|x+b|), \gamma_N]<0$ is confined within a tube of radius $\delta_N \sim 1/\sqrt{\log N}$,
and $i[V_{-1}(|x+b|), \gamma_N] \geq -4 f(\infty) |V_{-1}'(|x+b|)|$ on this region. Thus we proved the proposition.
\end{proof}

Next we prove that $i[V_j,\gamma_N]$ can be controlled by $i[-\Delta, \gamma_N]$, by using frequency decomposition.

\begin{thm} \label{thm:control}
Assume $V_{-1}(|x+b|)$ and $V_1(|x-b|)$ are real valued radially decreasing smooth potentials of compact support, with centers at $x=-b$ and $x=b$ respectively.
For any $\epsilon > 0$, there exists $N_{\epsilon}$ such that for any $N \geq N_{\epsilon}$, in the sense of forms on $D(-\Delta)$,
\begin{equation}\label{eqn:V_-1}
i[- \epsilon \Delta, \gamma_{-b}] +i[V_{-1}(|x+b|), \gamma_N] \geq -C\epsilon \frac{1}{\langle x+b \rangle^{\sigma}} \Delta \frac{1}{\langle x+b \rangle^{\sigma}} + i[V_{-1}(|x+b|),\gamma_N] >0,
\end{equation}
\begin{equation}\label{eqn:V_1}
i[- \epsilon \Delta, \gamma_{b}] +i[V_{1}(|x-b|), \gamma_N] \geq -C\epsilon \frac{1}{\langle x-b \rangle^{\sigma}} \Delta \frac{1}{\langle x-b \rangle^{\sigma}} + i[V_1(|x-b|),\gamma_N] >0.
\end{equation}
Here $C = 4-\sigma/(n-2)^2$, from Proposition \ref{prop:Delta}.
So, we have $i[-\epsilon \Delta,\gamma_N] + i[V_{-1}+V_1,\gamma_N] >0$.
\end{thm}

\begin{proof}
To prove the estimate involving $V_{-1}$, we only have to prove that for any $\epsilon >0$,
\begin{equation}- C \epsilon  \Delta - 4 f(\infty) \langle x+b \rangle^{\sigma} |\frac{\partial V_{-1}(|x+b|)}{\partial r}| \langle x+b \rangle^{\sigma} \chi^{(-1)}_N (x) > 0\end{equation}
for $N$ large enough.

For $j=-1, 1$, let us write
\begin{equation}
V_{a,\sigma}^{(j)}(x) = \frac{4}{C} f(\infty) \langle x-jb \rangle^{\sigma} |V_{j}'(|x-jb|)| \langle x-jb \rangle^{\sigma} \geq 0,
\end{equation}
and take $K = K_{\epsilon} = 2 \max_{j=-1,1} \{\max_{x\in\mathbb{R}^n} V_{a,\sigma}^{(j)}(x) \}/\epsilon$. Then we only need to prove
\begin{equation}
 -\epsilon \Delta - V_{a,\sigma}^{(-1)}(x) \chi^{(-1)}_N(x) \geq 0.
\end{equation}

We fix $I(\lambda)$ to be a smoothed characteristic function of $[0,\infty)$,
and fix $P(\lambda)$ be a smoothed characteristic function of $[3,\infty)$,
such that $0\leq I'(\lambda), P'(\lambda) \leq 1$ and $\supp P'(\lambda) \subset [1,3]$.
We take $P_K(\lambda) = P(\lambda/K)$ and $Q_K(\lambda) = I(\lambda/K) - P(\lambda/K)$.
Then by the spectral theorem, $P_K (-\Delta) + Q_K(-\Delta) =I(-\Delta/K) = \Id$, since $-\Delta$ is a positive operator and $I(x/K) = 1$ on $[0,\infty)$.
In the following context, for convenience, we will use $P_K$ and $Q_K$ to stand for $P_K(-\Delta)$ and $Q_K(-\Delta)$ respectively.
Thus we have
\begin{equation}\begin{split}
& -\epsilon \Delta - V_{a,\sigma}^{(-1)}(x) \chi^{(-1)}_N(x) \\
=& \epsilon (P_K + Q_K) (-\Delta) (P_K + Q_K) \\
& - (P_K + Q_K) V_{a,\sigma}^{(-1)}(x) \chi^{(-1)}_N(x) (P_K + Q_K)\\
=& \epsilon P_K (-\Delta) P_K + \epsilon Q_K (-\Delta) Q_K + \epsilon P_K (-\Delta) Q_K + \epsilon Q_K (-\Delta) P_K\\
& - P_K V_{a,\sigma}^{(-1)}(x) \chi^{(-1)}_N (x) P_K - Q_K V_{a,\sigma}^{(-1)}(x) \chi^{(-1)}_N(x) Q_K \\
& - P_K V_{a,\sigma}^{(-1)}(x) \chi^{(-1)}_N (x) Q_K - Q_K V_{a,\sigma}^{(-1)}(x) \chi^{(-1)}_N(x) P_K \\
\geq & \epsilon P_K (-\Delta) P_K + \epsilon Q_K (-\Delta) Q_K\\
& - 2P_K V_{a,\sigma}^{(-1)}(x) \chi^{(-1)}_N (x) P_K - 2Q_K V_{a,\sigma}^{(-1)}(x) \chi^{(-1)}_N(x) Q_K \\
= & P_K (-\epsilon \Delta-2V_{a,\sigma}^{(-1)}(x) \chi^{(-1)}_N(x)) P_K + Q_K (-\epsilon \Delta-2V_{a,\sigma}^{(-1)}(x) \chi^{(-1)}_N(x)) Q_K.
\end{split}\end{equation}
Here we used the Cauchy-Schwartz inequality and the fact that $P_K(\lambda) \lambda Q_K(\lambda) = Q_K(\lambda) \lambda P_K(\lambda) \geq 0$.

For the high frequency part,
\begin{equation}P_K (-\epsilon \Delta-2V_{a,\sigma}^{(-1)}(x) \chi^{(-1)}_N(x)) P_K \geq P_K (K\epsilon -2V_{a,\sigma}^{(-1)}(x) \chi^{(-1)}_N(x)) P_K \geq 0\end{equation}
by our choice of $K = 2 \max_{j=-1,1} \{\max_{x\in\mathbb{R}^n} V_{a,\sigma}^{(j)}(x) \}/\epsilon$.

For the low frequency part, we can apply the Hardy's inequality and get
\begin{equation}
\begin{split}
& Q_K (-\epsilon \Delta - 2 V_{a,\sigma}^{(-1)}(x) \chi^{(-1)}_N(x)) Q_K \\
=& Q_K (-\epsilon \Delta - 2 \widetilde{Q}_K V_{a,\sigma}^{(-1)}(x) \chi^{(-1)}_N(x) \widetilde{Q}_K) Q_K \\
\geq & Q_K (\epsilon \frac{(n-2)^2}{4|x|^2} -2 \widetilde{Q}_K V_{a,\sigma}^{(-1)}(x) \chi^{(-1)}_N(x) \widetilde{Q}_K) Q_K \\
=& Q_K \frac{1}{|x|} (\epsilon \frac{(n-2)^2}{4} -2 |x| \widetilde{Q}_K \chi^{(-1)}_N(x) V_{a,\sigma}^{(-1)}(x) \chi^{(-1)}_N(x) \widetilde{Q}_K |x|) \frac{1}{|x|} Q_K.
\end{split}
\end{equation}
Here we choose $\widetilde{Q}_K = Q_{3K}$, so $\widetilde{Q}_K Q_K = Q_K$.

To prove that the low frequency part is also positive, we need the following lemma:
\begin{lemma} \label{lemma:compact}
For fixed $K$, and $Q_K$ (or $\widetilde{Q}_K$) as above,
\begin{equation}
\| \chi^{(-1)}_N(x) Q_K|x| \| \to 0, \text{~as~} N \to \infty
\end{equation}
\end{lemma}

\begin{proof}[Proof of Lemma:]
First we prove that $\chi^{(-1)}(x) Q_K|x|$ is a compact operator; here $\chi^{(-1)} (x)$ is the characteristic function of $\supp V_{-1} (|x+b|)$.
We can either compute the integral kernel of this operator or use the commutation technique; here we will use the commutation technique.
We only have to prove $\| \chi^{(-1)}_N(x) Q_K|x_j| \| \to 0, \text{~as~} N \to \infty$,
because of the fact that $|x| = \frac{|x|}{\sum_j |x_j|} \sum_j |x_j|$ and that $\frac{|x|}{\sum_j |x_j|} \leq 1$.

\begin{equation}
\begin{split}
\chi^{(-1)}(x) Q_K|x_j| &= \chi^{(-1)}(x)(x_j + i) \frac{1}{x_j + i} Q_K |x_j| \\
 &= \chi^{(-1)}(x)(x_j +i) Q_K \frac{1}{x_j +i} |x_j| + \chi^{(-1)}(x)(x_j +i) [\frac{1}{x_j +i}, Q_K] |x_j| \\
&= \chi^{(-1)}(x)(x_j +i) Q_K \frac{|x_j|}{x_j +i} - \chi^{(-1)}(x) [x_j, Q_K] \frac{|x_j|}{x_j +i} \\
&= \chi^{(-1)}(x)(x_j +i) Q_K \frac{|x_j|}{x_j +i} - i \chi^{(-1)}(x) Q_K' 2p_j \frac{|x_j|}{x_j +i},
\end{split}
\end{equation}
where $p_j \equiv -i \partial_j$ are the momentum operators. We see that both terms are compact operators, so $\chi^{(-1)}(x) Q_K|x_j|$ is compact.
And we have $\chi^{(-1)}_N(x)$ goes to $0$ strongly, so 
\begin{equation}
\| \chi^{(-1)}_N(x) Q_K|x_j| \| = \| \chi^{(-1)}_N(x) \chi^{(-1)}(x) Q_K|x_j| \| \to 0,
\end{equation}
as $N$ goes to infinity. Thus the lemma is proved.
\end{proof}
With the help of the lemma, we get $Q_K (-\epsilon \Delta - 2 V_{a,\sigma}^{(-1)}(x) \chi^{(-1)}_N(x)) Q_K$ is positive for $N$ large enough.

So we proved the equation (\ref{eqn:V_-1}), and similarly one can prove equation (\ref{eqn:V_1}).
\end{proof}

Then Proposition \ref{prop:Delta} and Theorem \ref{thm:control} together imply Theorem \ref{thm:two-bump}.

\section{One Dimensional Lattice Potential}

Using the same method, we prove similar result when the potential consists of bump functions centered at one dimensional lattice points.
Without loss of generality, we can assume that $H= -\Delta + \sum_{j=-M}^{M} V_j(|x-b_j|)$, where $b_j = (j,\vec{0}) \in \mathbb{R}^n$, $j=-M,\dots,M$.
Assume $V_{j}(|x- b_j|)$ are real valued radially decreasing smooth potentials of compact support, with centers at $x=b_j$.
And we define $\gamma_N = \sum_{k= -N}^{N} \gamma_{c_k}$, with $c_k = (k,\vec{0}) \in \mathbb{R}^n$, $k = -N,\dots,N$. Then we have the following theorem:

\begin{thm}\label{thm:latice}
For $H= -\Delta + \sum_{j=-M}^{M} V_j(|x-b_j|)$, where $b_j = (j,\vec{0}) \in \mathbb{R}^n$, $j=-M,\dots,M$.
Assume $V_{j}(|x- b_j|)$ are real valued radially decreasing smooth potentials of compact support, with centers at $x=b_j$.
If $1/2 < \sigma < 4(n-2)^2$, 
then for any $0<\epsilon<1$, there exists $N_{\epsilon} \geq M$, such that for any $N \geq N_{\epsilon}$,
\begin{equation}
i[H,\gamma_N] \geq (1 -\epsilon) i[-\Delta, \gamma_N].
\end{equation}
\end{thm}

The proof is the same as the two-bump potential except for a few points.

First, for fixed $\epsilon >0$, the cut off of energy (frequency) should be
\begin{equation}
 K = 2 \max_{j=-M,\dots,M} \{\max_{x\in\mathbb{R}^n} V_{a,\sigma}^{(j)}(x) \}/\epsilon.
\end{equation}
Here the functions $V_{a,\sigma}^{(j)}(x)$ are defined as
\begin{equation}
V_{a,\sigma}^{(j)}(x) = 2 M f(\infty) \langle x-b_j \rangle^{\sigma} |V_j'(|x-b_j|)| \langle x-b_j \rangle^{\sigma} \geq 0
\end{equation}
If all the potential functions are of uniform shape, then the cut off is the same as the two-bump case.

Second, for each $i[V_j, \gamma_N]$, we may have $2M$ possibly negative terms, instead of only two terms in the two-bump case, after combining symmetric $\gamma$'s using the cancellation lemma. And of course, we have $M$ such $V$'s. This affect the result in two ways:
\begin{enumerate}
\item The rate of convergence for $\delta_N \to 0$ became slower,
which eventually slows the rate of convergence for $\| \chi^{(j)}_N(x) Q_K|x| \| \to 0$ as $N \to \infty$.
But we still have $i[-\epsilon \Delta, \gamma_{c_j}] + i[V_j,\gamma_N]>0$, by  using a larger value of $N$.
\item We need to control the commutators of $\gamma_N$ with $M$ bump functions instead of 2.
But remember in the two-bump potential case, we only used $i[-\epsilon \Delta, \gamma_{-b} + \gamma_{b}]$.
If we utilize all the terms in $i[-\epsilon \Delta, \gamma_N]$ that are localized at the position of the bumps,
we still have
\begin{equation}
 i[-\epsilon\Delta,\gamma_N] + \sum_{j=-M}^{M} i[V_j,\gamma_N] \geq \sum_{j=-M}^M i[-\epsilon \Delta, \gamma_{c_j}] + i[V_j,\gamma_N] >0.
 \end{equation}
\end{enumerate}
So we still have the same result, but with a larger value of $N$ (but still finite).

\section{Axially Repulsive Potentials}

As one can see, in the proof we essentially used the fact that the potential funtion $V(x)$ is repelling in all directions except one,
say $x_1$ direction.
Also we need the repelling force to be strong enough outside a tube near 0, 
so that we can shrink the region where $i[V(x), \gamma_N] < 0$ to a tube as small as we want by increasing $N$.
Then we prove the same result for a larger class of $V(x)$.

In this section, we will often use the notation $x = (x_1, \vec{y}) \in \mathbb{R}^n$, with $\vec{y} \in \mathbb{R}^{n-1}$.

We assume that the potential function $V(x)$ satisfies the following properties:
\begin{enumerate}
\item[(A1)] $V(x)$ is $C^1$ and non-negative.

\item[(A2)] $V(x)$ is axially repulsive:
\begin{enumerate}
 \item
$V(x)$ is axially repulsive w.r.t. the $x_1$-axis, that is, for any $x = (x_1, \vec{y}) \in \mathbb{R} \times \mathbb{R}^{n-1}$,
$$- \vec{y} \cdot \nabla_{\vec{y}} V(x) \geq 0.$$
\item For $x \in \{x = (x_1, \vec{y}): |x_1|>L \}$, $V(x)$ is also repulsive in the $x_1$ direction , that is
$$
- x_1 \cdot \partial_{x_1 }V(x) \geq 0, \text{~and~} - \vec{y} \cdot \nabla_{\vec{y}} V(x) \geq 0.
$$
\end{enumerate}

\item[(A3)] $|\partial_{x_1} V|$ can be controlled by $|\nabla_{\vec{y}} V(x) \cdot \frac{\vec{y}}{|\vec{y}|}|$, 
in the region $x \in [-L,L] \times \mathbb{R}^{n-1}$:
for any $0< \delta <L$, there exists $\Lambda_{\delta}>0$ such that for any $x \in [-L, L] \times (\mathbb{R}^{n-1}/B^{n-1}_0(\delta))$
$$|\frac{\partial V(x)}{\partial x_1}| \leq \Lambda_{\delta} |\nabla_{\vec{y}} V(x) \cdot \frac{\vec{y}}{|\vec{y}|}|.$$
Here, $B_0^{n-1}(R)$ is the $(n-1)$-dimensional ball centered at 0 with radius $R$.
\end{enumerate}

For the applications, sometimes we need a slightly stronger condition than (A2-b):
\begin{enumerate}
\item[(A4)] For $x \in \{x = (x_1, \dots, x_n): |x_1|>L \}$, $V(x)$ is repulsive in every $x_j$ direction , that is
$$
- x_j \cdot \partial_{x_j }V(x) \geq 0,
$$
for $j = 1, \dots, n$.
\end{enumerate}

\begin{remark}
We can choose $\Lambda_{\delta}$ to be nonincreasing in $\delta$.
Normally, $\Lambda_{\delta} \to \infty$ as $\delta$ goes to 0, 
and we fix $\Lambda_{\delta} = \Lambda_R$ when $\delta>R$ for some large $R$.
\end{remark}

\begin{example}\label{example}
The following two examples satisfy the properties (A1-A4).
\begin{enumerate}
 \item In the one dimensional lattice case, $V(x) = \sum_{j=-M}^{M} V_j(|x-b_j|)$,
 with $b_j = (j,\vec{0})$ and $V_j(r)\geq 0$ is decreasing and compactly supported smooth function.
 Then we can take $L=M$, and $\Lambda_{\delta} = \max \{1, M/\delta\}$.
 \item $V(x) = X(x_1) e^{-|\vec{y}|}$, with $x = (x_1, \vec{y}) \in \mathbb{R} \times \mathbb{R}^{n-1}$. $X(x_1)$ is non-negative and smooth function.
 Suppose $X'(x_1) \leq K X(x_1)$ when $x_1 \in [-L,L]$, and $X(x_1)$ is repulsive when $x_1 \notin [-L, L]$.
 Then we can take $\Lambda_{\delta} = \max \{K, KL/\delta\}$.
\end{enumerate}
\end{example}

\begin{remark}
 It should be noted that if $V(x) = X(x_1) |y|^{-(3+\epsilon)}$, with $X(x_1)$ as in the Example \ref{example}, 
 then $V(x)$ does not satisfy the condition (A3).
 But one can still prove that $i[-\Delta + V(x), \gamma_N]$ will be a positive operator for $N$ large enough.
 This is because one can view $V(x) = V_1(x) + V_2(x)$, with $V_1(x)$ satisfies the properties (A1-A4),
 while $V_2(x)$ is a perturbation that can be controlled by the Laplacian term.
\end{remark}

Then with such $V(x)$, we prove the same result as the one dimensional lattice case.

\begin{thm} \label{thm:general}
Let $H= -\Delta + V(x)$, with $V(x)$ satisfies the properties (A1-A3).
Define the multiplier $\gamma_N = \sum_{k=-N}^{N} \gamma_{c_k}$, where $c_k = k = (k,\vec{0}) \in \mathbb{R}^n$.
If $1/2 < \sigma <4(n-2)^2$,
then for any $0<\epsilon<1$, there exists $N_{\epsilon}$, such that for any $N \geq N_{\epsilon}$,
\begin{equation}
\begin{split}
i[H,\gamma_N] &\geq (1 -\epsilon) i[-\Delta, \gamma_N] \\
&\geq -(1- \epsilon)(4 - \frac{\sigma}{(n-2)^2}) \sum_{k=-N}^N \frac{1}{\langle x- k \rangle^{\sigma}} \Delta \frac{1}{\langle x- k \rangle^{\sigma}}.
\end{split}
\end{equation}
\end{thm}

To prove Theorem \ref{thm:general}, we will need a more general cancellation lemma than the one we used for two-bump potential,
because $V(x)$ does not have radial symmetry now.

\begin{lemma}[General Cancellation Lemma] \label{lemma:general-cancellation}
Let $V, \gamma_N$ as above, and define
$$S_V^+ = \{x=(x_1,\vec{y}): -\vec{y}\cdot \nabla_{\vec{y}} V \geq 0 \text{~and~} -x_1 \cdot \partial_{x_1} V \geq 0 \}$$
be the region where $V(x)$ is repulsive both in the $x_1$ and $\vec{y}$ direction, then we have
\begin{enumerate}
\item If $x\in S_V^+$, then $i[V(x), \gamma_{-c} + \gamma_c] \geq 0$ for any $c = (c_1,\vec{0}) \in \mathbb{R}^n$.
\item For any $N$ and any $x \in \mathbb{R}^{n}$,
\begin{equation}
i[V(x), \gamma_N] \geq - |\partial_{x_1} V(x)| (2L+3) f(\infty) -2 \nabla_{\vec{y}} V(x) \cdot \vec{y} \sum_{k=-N}^{N} \frac{f(|x-k|)}{|x-k|}. 
\end{equation}
\item For any $\delta$, there exists $N_{\delta}$ such that for all $N\geq N_{\delta}$,
we have 
\begin{equation}
i[V(x), \gamma_N] \geq - |\partial_{x_1} V(x)| (2L+3) f(\infty) \chi_{\delta}.
\end{equation}
Here $\chi_{\delta}$ is the characteristic function of $[-L,L] \times B_0^{n-1}(\delta)$.
Quantitatively, $N_{\delta} \sim e^{L \Lambda_{\delta}/\delta}$.
Particularly, in the one dimensional lattice case (with radial symmetric potentials), $N_{\delta} \sim e^{L^2/\delta^2}$.
\end{enumerate}
\end{lemma}

The General Cancellation Lemma says that the region where $i[V(x),\gamma_N]<0$ will shrink to a small tube as $N$ goes to infinity,
and that there is a uniform lower bound for $i[V(x),\gamma_N]$.
These results are essential to our proof of Theorem \ref{thm:general}.

The key idea of proving the lemma is that
there is cancellation for (possibly) negative terms coming from the non-repulsive effect in $x_1$ direction, 
if we pair up $\gamma_c$'s in a proper way.
Then all the gain from the repulsive effect in $\vec{y}$ direction will accumulate and go to infinity as $N$ go to infinity.
So the positive terms will eventually dominate the (possibly) negative terms and then shrink the negative region to a small tube.

\begin{proof}[Proof of Lemma (\ref{lemma:general-cancellation})]
First we compute $i[V(x),\gamma_N]$:
\begin{equation}\label{eqn:cancellation}
\begin{split}
i[V(x),\gamma_N] =& -2 \sum_{k=-N}^{N} \frac{\partial V}{\partial x} \cdot \frac{\partial F_{c_k}}{\partial x} \\
=& -2 \frac{\partial V(x)}{\partial x_1} \sum_{k=-N}^{N} \frac{f(|x-k|)}{|x-k|} (x_1 - k) \\
&-2 \nabla_{\vec{y}} V(x) \cdot \vec{y} \sum_{k=-N}^{N} \frac{f(|x-k|)}{|x-k|} \\
=& I + II.
\end{split}\end{equation}
We know the second term on the RHS of (\ref{eqn:cancellation}) is always non-negative by the axially repulsive property of $V(x)$,
so the key point is to control the first term using cancellation.

We need the following claim to estimate part $I$.

\begin{claim}\label{claim:cancellation}
For any real numbers $k_1$ and $k_2$ and some fixed $\vec{y}_0 \in \mathbb{R}^{n-1}$, write $\vec{k}_j = (k_j, \vec{y}_0)$, for $j = 1, 2$.
If $k_1 \leq x_1 \leq k_2$, and $|k_1 - x_1| \leq |k_2 - x_1|$, then
\begin{equation}
\frac{f(|x-\vec{k}_1|)}{|x-\vec{k}_1|} (x_1 - k_1) + \frac{f(|x-\vec{k}_2|)}{|x-\vec{k}_2|} (x_1 - k_2) \leq 0.
\end{equation}
If, on the other hand, $k_1 \leq x_1 \leq k_2$, and $|k_1 - x_1| \geq |k_2 - x_1|$, then
\begin{equation}
\frac{f(|x-\vec{k}_1|)}{|x-\vec{k}_1|} (x_1 - k_1) + \frac{f(|x-\vec{k}_2|)}{|x-\vec{k}_2|} (x_1 - k_2) \geq 0.
\end{equation}
\end{claim}

\begin{proof}[Proof of Claim \ref{claim:cancellation}]
The proof follows from the fact that
\begin{equation}\label{eqn:increasing}
\begin{split}
\frac{f(|x-\vec{k}_j|)}{|x-\vec{k}_j|} |x_1 - k_j| 
=& f(\sqrt{|x_1 - k_j|^2 + |\vec{y} - \vec{y}_0|^2}) \frac{|x_1 - k_j|}{\sqrt{|x_1 - k_j|^2 + |\vec{y} - \vec{y}_0|^2}} \\
=& f(\sqrt{|x_1 - k_j|^2 + |\vec{y} - \vec{y}_0|^2}) \sqrt{1 - \frac{|\vec{y} - \vec{y}_0|^2}{|x_1 - k_j|^2 + |\vec{y} - \vec{y}_0|^2}}
\end{split}
\end{equation}
is an increasing function of $z=|x_1 - k_j|$, because both of the two factors on the RHS of (\ref{eqn:increasing}) are increasing functions of $z$.
\end{proof}

Suppose $x =(x_1, \vec{y}) \in S_V^+$, then, for $c= (c_1,\vec{0})$ with $c_1 > 0$,
\begin{equation}
 \begin{split}
i[V(x),\gamma_{-c} + \gamma_{c}] =& -2 \frac{\partial V(x)}{\partial x_1} (\frac{f(|x+c|)}{|x+c|} (x_1 + c_1) + \frac{f(|x-c|)}{|x-c|} (x_1 - c_1) )\\
&-2 \nabla_{\vec{y}} V(x) \cdot \vec{y} (\frac{f(|x+c|)}{|x+c|} + \frac{f(|x-c|)}{|x-c|}).
 \end{split}
\end{equation}

If $x_1 \leq -c_1$ or $x_1 \geq c_1$, then clearly $i[V(x),\gamma_{-c} + \gamma_{c}] \geq 0$.

If $-c_1 < x_1 < c_1$, then use Claim \ref{claim:cancellation}, we still have $i[V(x),\gamma_{-c} + \gamma_{c}] \geq 0$.
This is because, if $-c_1 < x_1 \leq 0$, then $|x_1 + c_1| \leq |x_1 -c_1|$, thus we have 
$$\frac{f(|x-c_1|)}{|x-c_1|} (x_1 - c_1) + \frac{f(|x-c_2|)}{|x-c_2|} (x_1 - c_2) \leq 0.$$
We also have $-\partial_{x_1} V(x) \leq 0$, so $i[V(x),\gamma_{-c} + \gamma_{c}] \geq 0$ if $-c_1 < x_1 \leq 0$.
Similarly, one can prove it for $0 \leq x_1 \leq c_1$ using Claim \ref{claim:cancellation}.

If $x \notin [-L,L] \times \mathbb{R}^{n-1}$, then $x$ is repulsive both in the $x_1$ and $\vec{y}$ directions, i.e. $x \in S_V^+$. 
Then $i[V(x), \gamma_{-c}+\gamma_c] \geq 0$ for any $c = (c_1, \vec{0}) \in \mathbb{R}^n$.
By pairing up the symmetric $\gamma$'s in $\gamma_N$, we get
\begin{equation}
 i[V(x), \gamma_N] \geq 0
\end{equation}
for $x \notin [-L,L] \times \mathbb{R}^{n-1}$.

In the computation below, we only consider the region $x \in [-L,L] \times \mathbb{R}^{n-1}$.

The Claim \ref{claim:cancellation} enables us to pair up the $\gamma_c$'s based on the sign of $\partial V(x)/\partial x_1$,
so that $i[V(x), \gamma_{c_{k_1}} + \gamma_{c_{k_2}}] \geq 0$ at $x$.

If $\partial_{x_1} V(x) \geq 0$, then we pair up $\gamma_c$'s centered at $x= \lfloor x_1 \rfloor -k+1$ and $x= k+\lceil L \rceil$:
\begin{equation}\begin{split}
& -2 \frac{\partial V(x)}{\partial x_1} \big\{ \sum_{j = \lfloor x_1 \rfloor - (N- \lceil L \rceil -1)}^{\lfloor x_1 \rfloor} \frac{f(|x-j|)}{|x-j|} (x_1 - j) + \sum_{k = \lceil L \rceil + 1}^{N} \frac{f(|x-k|)}{|x-k|} (x_1 - k) \big\} \\
=& -2 \frac{\partial V(x)}{\partial x_1} \sum_{k = 1}^{N-\lceil L \rceil} \big\{ \frac{f(|x-(k+\lceil L \rceil)|)}{|x-(k+\lceil L \rceil)|} (x_1 - (k+\lceil L \rceil)) \\
& + \frac{f(|x-(\lfloor x_1 \rfloor -k+1)|)}{|x-(\lfloor x_1 \rfloor -k+1)|} (x_1 - (\lfloor x_1 \rfloor -k+1)) \big\}\\
\geq & 0.
\end{split}\end{equation}
This is because, $\lfloor x_1 \rfloor -k+1 \leq x_1 \leq L \leq k+\lceil L \rceil$, and $|x_1 - (\lfloor x_1 \rfloor -k+1)| < k \leq |x_1 - (k+\lceil L \rceil)|$.
Then use the claim we just stated, for each $k$,
\begin{equation}
\begin{split}
- \frac{\partial V(x)}{\partial x_1} \big\{ & \frac{f(|x-(k+\lceil L \rceil)|)}{|x-(k+\lceil L \rceil)|} (x_1 - (k+\lceil L \rceil)) \\
& + \frac{f(|x-(\lfloor x_1 \rfloor -k+1)|)}{|x-(\lfloor x_1 \rfloor -k+1)|} (x_1 - (\lfloor x_1 \rfloor -k+1)) \big\} \geq 0.
\end{split}
\end{equation}

Similarly, if $\partial_{x_1} V(x)\leq 0$, we get
\begin{equation}
-2 \frac{\partial V(x)}{\partial x_1} \big\{ \sum_{j = \lceil x_1 \rceil}^{\lceil x_1 \rceil + (N- \lceil L \rceil -1)} \frac{f(|x-j|)}{|x-j|} (x_1 - j)
+ \sum_{k = -N}^{-(\lceil L \rceil + 1)} \frac{f(|x-k|)}{|x-k|} (x_1 - k) \big\} \geq 0.
\end{equation}

After the pairing, we have at most $2N+1 - 2(N-\lceil L \rceil) = 2 \lceil L \rceil +1 \leq 2L+3$ terms left. So we proved that, for each $x \in \supp V$,
\begin{equation}
I = -2 \frac{\partial V(x)}{\partial x_1} \sum_{k=-N}^{N} \frac{f(|x-k|)}{|x-k|} (x_1 - k) \geq - 2|\frac{\partial V(x)}{\partial x_1}| (2L+3) f(\infty).
\end{equation}

To estimate part $II$, take $R > \lceil 2\Lambda_{L} (2L+3) \rceil +L$ and
\begin{equation}
 f(R) > \frac{1}{2} \frac{M_{\sigma}}{\sqrt{a}} = \frac{1}{2} f(\infty).
\end{equation}
Let $N > R-L$.

For $x \in [-L,L] \times (\mathbb{R}^{n-1}/B_0^{n-1} (R))$ and $|k| < \lfloor R-L \rfloor$, we have
\begin{equation}
 \frac{|\vec{y}|}{|x-k|} = \frac{|\vec{y}|}{\sqrt{|x_1 - k|^2 + |\vec{y}|^2}} \geq \frac{1}{\sqrt{2}} > \frac{1}{2}.
\end{equation}
So the positive contribution of $i[V(x),\gamma]$, i.e. part $II$, will be
\begin{equation}
\begin{split}
II = & -2 (\nabla_{\vec{y}} V(x) \cdot \vec{y})
\sum_{k=-N}^{N} \frac{f(|x-k|)}{|x-k|}\\
\geq & -2 (\nabla_{\vec{y}} V(x) \cdot \frac{\vec{y}}{|\vec{y}|})
\sum_{k=-\lfloor R-L \rfloor}^{\lfloor R-L \rfloor} \frac{|\vec{y}|}{|x-k|} f(|x-k|)\\
\geq & -2 (\nabla_{\vec{y}} V(x) \cdot \frac{\vec{y}}{|\vec{y}|})
(\frac{2\lfloor R-L \rfloor}{4} f(\infty)) \\
\geq & -2 (\nabla_{\vec{y}} V(x) \cdot \frac{\vec{y}}{|\vec{y}|})
(\Lambda_L (2L+3) f(\infty)).
\end{split}
\end{equation}
By the property (A3) of $V(x)$ and the estimation of part $I$,
we get $i[V(x),\gamma] \geq 0$ on $x \in [-L,L] \times (\mathbb{R}^{n-1}/B_0^{n-1} (R))$. And we only need $N \sim (\Lambda_L +1) L$.

For $x \in [-L,L] \times (B_0^{n-1} (R))/B_0^{n-1} (\delta)$, $\sum_{k=-N}^{N} \frac{f(|x-k|)}{|x-k|} \sim f(\infty) \log N $.
And by property (A3) of $V(x)$,
$|\frac{\partial V(x)}{\partial x_1}| \leq \Lambda_{\delta} |\nabla_{\vec{y}} V(x) \cdot \frac{\vec{y}}{|\vec{y}|}|$.
Combined with the estimation of part $I$,
we prove that for any $\delta>0$, there exists $N_{\delta} \sim e^{L \Lambda_{\delta}}$,
such that for all $N\geq N_{\delta}$ we have $i[V(x), \gamma_N] \geq 0$ if $x \in [-L,L] \times (B_0^{n-1} (R))/B_0^{n-1} (\delta)$.
\end{proof}

Now we are ready to prove the Theorem (\ref{thm:general}).

\begin{proof}[Proof of Theorem (\ref{thm:general})]
As before, we have
\begin{equation}
i[-\Delta,\gamma_N] \geq -(4-\frac{\sigma}{(n-2)^2}) \sum_{k=-N}^{N} \frac{1}{\langle x-k \rangle^{\sigma}} \Delta \frac{1}{\langle x-k \rangle^{\sigma}}.
\end{equation}

The general cancellation lemma tells us that for any $\delta>0$, there exists $N_{\delta}$, such that for all $N > N_{\delta}$
\begin{equation}
i[V(x),\gamma_N] \geq - 2|\frac{\partial V(x)}{\partial x_1}| \chi_{\delta}(x) (2L+3) f(\infty).
\end{equation}
Here, $\chi_{\delta}(x)$ is the characteristic function of the small tube $[-L,L] \times B_0^{n-1}(\delta)$.

We follow the same scheme of the proof for Theorem (\ref{thm:control}). For any $\epsilon >0$, we want to prove that for $N$ sufficiently large,
\begin{equation}
i[-\epsilon \Delta,\gamma_N] + i[V(x),\gamma_N] \geq 0.
\end{equation}
To get this estimate, we only have to prove that for $\delta>0$ small enough,
\begin{equation}
-\epsilon \frac{1}{\langle x \rangle^{\sigma}} \Delta \frac{1}{\langle x \rangle^{\sigma}} - 2|\frac{\partial V(x)}{\partial x_1}| \chi_{\delta}(x) (2L+3) f(\infty) \geq 0.
\end{equation}
Now define $\widetilde{V}(x) \equiv 2 \langle x \rangle^{\sigma} |\frac{\partial V(x)}{\partial x_1}| \langle x \rangle^{\sigma} (2L+3) f(\infty)$,
so we only need to prove that, given $\epsilon >0$, there exists $\delta>0$ sufficiently small,
\begin{equation}
-\epsilon \Delta - \widetilde{V}(x) \chi_{\delta}(x) \geq 0.
\end{equation}

We take $K = K_{\epsilon} = 2\max_{x\in\mathbb{R}^n} \widetilde{V}(x) \chi_{\delta_0} /\epsilon$, for some fixed $\delta_0 >0$.
And define $P_K$, $Q_K$ the same way as in Theorem (\ref{thm:control}).
Similarly, we have
\begin{equation}
 -\epsilon \Delta - \widetilde{V}(x) \chi_{\delta}(x)
\geq P_K (-\epsilon \Delta - 2\widetilde{V}(x) \chi_{\delta}(x)) P_K + Q_K (-\epsilon \Delta - 2\widetilde{V}(x) \chi_{\delta}(x)) Q_K.
\end{equation}

For the high frequency part,
\begin{equation}
P_K (-\epsilon \Delta - 2\widetilde{V}(x) \chi_{\delta}(x)) P_K \geq P_K (-\epsilon K - 2\widetilde{V}(x) \chi_{\delta}(x)) P_K \geq 0.
\end{equation}

For the low frequency part, similar to Theorem (\ref{thm:control}), we have
\begin{equation}
\begin{split}
& Q_K (-\epsilon \Delta - 2 \widetilde{V}(x) \chi_{\delta}(x)) Q_K \\
\geq & Q_K \frac{1}{|x|} (\epsilon \frac{(n-2)^2}{4} - 2 |x| \widetilde{Q}_K \chi_{\delta}(x) \widetilde{V}(x) \chi_{\delta}(x) \widetilde{Q}_K |x|) \frac{1}{|x|} Q_K.
\end{split}
\end{equation}
Here $\widetilde{Q}_K = Q_{3K}$, so $\widetilde{Q}_K Q_K = Q_K$.

We want to estimate the norm of $\chi_{\delta}(x) \widetilde{Q}_K |x|$. As in Lemma (\ref{lemma:compact}), we have
\begin{equation}
\begin{split}
\chi_{\delta}(x) \widetilde{Q}_K |x_j| &= \chi_{\delta}(x)(x_j +i) \frac{1}{x_j +i} \widetilde{Q}_K |x_j| \\
&= \chi_{\delta}(x)(x_j +i) \widetilde{Q}_K \frac{|x_j |}{x_j +i} -i \chi_{\delta}(x) \widetilde{Q}'_K 2p_j \frac{|x_j |}{x_j +i}.
\end{split}
\end{equation}
So,
\begin{equation}
\begin{split}
\| \chi_{\delta}(x) \widetilde{Q}_K |x_j| \| &\leq \| \chi_{\delta}(x)(x_j +i) \widetilde{Q}_K \frac{|x_j|}{x_j +i}\| + \|\chi_{\delta}(x) \widetilde{Q}'_K 2p_j  \frac{|x_j|}{x_j +i}\| \\
&\leq \| \chi_{\delta}(x)(x_j +i) \widetilde{Q}_K (p^2) \|_{HS} + \|\chi_{\delta}(x) \widetilde{Q}'_K(p^2) 2p_j \|_{HS}.
\end{split}
\end{equation}
Here $p_j \equiv -i\partial_j$ are the momentum operators, and $\|\cdot\|_{HS}$ is the Hilbert-Schmidt norm.

We can compute the Hilber-Schmidt norms $\| \chi_{\delta}(x)(x_j +i) Q_K(p^2) \|_{HS} \sim C_n (L^3 K^{n/2} \delta^{n-1})^{1/2}$, and $\|\chi_{\delta}(x) Q_K'(p^2) 2p_j \|_{HS} \sim C_n (L K^{(n+2)/2} \delta^{n-1})^{1/2}$.
So, sum up the index $j$, we will get
\begin{equation}
\| \chi_{\delta}(x) Q_K|x| \| \lesssim C_n (L^3 K^{n/2} \delta^{n-1})^{1/2} + C_n (L K^{(n+2)/2} \delta^{n-1})^{1/2}.
\end{equation}
And thus we have
\begin{equation}
\| 2 |x| Q_K \chi_{\delta}(x) \widetilde{V}(x) \chi_{\delta}(x) Q_K|x| \| \lesssim C_n \epsilon K(L^3 K^{n/2} + L K^{(n+2)/2}) \delta^{n-1}.
\end{equation}
By choosing $\delta$ small enough, we have
\begin{equation}
Q_K (-\epsilon \Delta - 2 \widetilde{V}(x) \chi_{\delta}(x)) Q_K \geq 0.
\end{equation}
Actually, $\delta$ is approximately
\begin{equation}
\delta \sim \frac{C_n}{(L^3 K_{\epsilon}^{(n+2)/2} + L K_{\epsilon}^{(n+4)/2})^{1/(n-1)}}.
\end{equation}

Combining the high frequency part and low frequency part, we proved the theorem.
\end{proof}

\begin{cor}
For the one dimensional lattice case, for any $\epsilon>0$, the minimum $N_{\epsilon}$ required is approximately
\begin{equation}
N_{\epsilon} \sim C_n \exp(M^2/\delta^2) \sim C_n \exp(M^2 (M^3 K_{\epsilon}^{(n+2)/2} + M K_{\epsilon}^{(n+4)/2})^{2/(n-1)}).
\end{equation}
\end{cor}

\begin{remark}
The Lemma \ref{lemma:general-cancellation} and Claim \ref{claim:cancellation} are also true if we replace $\gamma_c$ by $\gamma_{c}^{Mor}$.
Because we basically only used the nondecreasing property of $f(|x|)$ in the proof, 
and $\gamma_c$ become $\gamma_{c}^{Mor}$ when $f\equiv 1$.
But to prove Theorem \ref{thm:general}, we need $i[-\Delta,\gamma_N]$ to absorb the negative region using the frequency decomposition.
\end{remark}

The next Theorem will be useful in proving interaction Morawetz estimate for $H = -\Delta + V(x)$.
First we need some notation:
for $c= (c_1, \dots, c_k) \in \mathbb{R}^k$, define
$$\sym \{c\} = \{(x_1,\dots,x_k): x_j = \pm c_j, j=1,\dots,k \},$$
the set of all symmetric points of $c$ in $\mathbb{R}^k$ (w.r.t. every $x_j$-axis).

\begin{thm} \label{thm:interaction-morawetz}
If $V(x)$ satisfies the conditions (A1-A4), and $\gamma_N$ as defined in Theorem \ref{thm:general}, 
then for any $\delta>0$, there exists $N_{\delta}$, such that
\begin{equation} \label{eqn:mor-tube}
i[V(x), \gamma_N + \sum_{c \in \sym\{x'\}} \gamma_{c}^{Mor}] \gtrsim -|\nabla V (x)| L f(\infty) \chi_{\delta},
\end{equation}
for any $N > N_{\delta}$ and any $x' \in \mathbb{R}^n$. Here $\chi_{\delta}$ is the characteristic function of $ [-L,L] \times B_0^{n-1}(\delta)$.
If $1/2 < \sigma < 4(n-2)^2$, then for any $0< \epsilon <1$, there exists $N_{\epsilon}$, such that
\begin{equation} \label{eqn:mor-positive}
i[H,\gamma_N + \sum_{c \in \sym\{x'\}} \gamma_{c}^{Mor}] \geq i[-\Delta,\sum_{c \in \sym\{x'\}} \gamma_{c}^{Mor}]
+ (1-\epsilon) i[-\Delta,\gamma_N],
\end{equation}
for $N > N_{\epsilon}$.
\end{thm}

The proof of Theorem \ref{thm:interaction-morawetz} is similar to the proof of Theorem \ref{thm:general} (and Lemma \ref{lemma:general-cancellation}),
the only thing new is the treatment for the region $x \notin [-L,L] \times \mathbb{R}^{n-1}$.
And that is the reason why we have a stronger condition on this region.

\begin{proof}[Proof of Theorem \ref{thm:interaction-morawetz}]
For the region $x \in [-L,L] \times (\mathbb{R}^{n-1} \setminus B_0^{n-1}(\delta))$, 
we have enough gain from $i[V, \gamma_N]$ to control the terms coming from $i[V, \sum_{c \in \sym\{x'\}} \gamma_{c}^{Mor}]$,
no matter where $x'$ is.

Suppose we have $i[V, \sum_{c \in \sym\{x'\}} \gamma_{c}^{Mor}] \geq 0$ when $x \notin [-L,L] \times \mathbb{R}^{n-1}$,
then we get (\ref{eqn:mor-tube}).
Thus equation (\ref{eqn:mor-positive}) follows by frequency decomposition.
So all we need to prove is:
\begin{lemma} \label{lemma:mor-symmetric}
 If $V(x)$ satisfies the condition (A4), then
\begin{equation}
i[V(x), \sum_{c \in \sym\{x'\}} \gamma_{c}^{Mor}] \geq 0,
\end{equation}
for $x \notin [-L,L] \times \mathbb{R}^{n-1}$.
\end{lemma}

\begin{proof}[Proof of Lemma \ref{lemma:mor-symmetric}:]
First, we have
\begin{equation}
  i[V(x), \sum_{c \in \sym\{x'\}} \gamma_{c}^{Mor}] 
  = \sum_{j=1}^n \Bigl\{ -2\partial_{x_j}V(x) \sum_{c \in \sym\{x'\}} \frac{x_j - c_j}{|x-c|} \Bigr\}
\end{equation}
We actually prove that, for each $j$, and $x \notin [-L,L] \times \mathbb{R}^{n-1}$,
\begin{equation}
  -2\partial_{x_j}V(x) \sum_{c \in \sym\{x'\}} \frac{x_j - c_j}{|x-c|} \geq 0,
\end{equation}
by pairing up $c \in \sym\{x'\}$ properly.

We only show the pairing for $j=1$, the rest are similar. Let $x' = (x'_1, \vec{y'})$, $c= (c_1,\vec{y}_c)$,
\begin{equation}
\begin{split}
 & -2\partial_{x_1}V(x) \sum_{c \in \sym\{x'\}} \frac{x_1 - c_1}{|x-c|} \\
=& -2\partial_{x_1}V(x) \sum_{\vec{y}_c \in \sym\{\vec{y'}\}} \frac{x_1 -(- c_1)}{|(x_1 - (- c_1), \vec{y} - \vec{y}_c)|}
+ \frac{x_1 - c_1}{|(x_1 - c_1, \vec{y} - \vec{y}_c)|}\\
=& \sum_{\vec{y}_c \in \sym\{\vec{y'}\}} -2\partial_{x_1}V(x) \Bigl \{  \frac{x_1 + c_1}{|(x_1 + c_1, \vec{y} - \vec{y}_c)|}
+ \frac{x_1 - c_1}{|(x_1 - c_1, \vec{y} - \vec{y}_c)|} \Bigr\}\\
\geq & 0,
\end{split}
\end{equation}
for $x \notin [-L,L] \times \mathbb{R}^{n-1}$.
The last step is a direct result of Claim \ref{claim:cancellation} and the repulsive condition (A4).
\end{proof}

Lemma \ref{lemma:mor-symmetric} then completes the proof of Theorem \ref{thm:interaction-morawetz}.
\end{proof}

\section{Application}
\subsection{Decay and Strichartz Estimates}
As an application of Theorem (\ref{thm:general}),
we prove the Strichartz estimates for Schr\"odinger operators with axially repulsive potentials in a way that extends to the defocusing NLS.

First, we prove the following estimates:

\begin{thm}\label{thm:decay}
 For dimension $n\geq 3$, suppose $H = -\Delta +V(x)$, with $V(x)$ axially repulsive, i.e. $V(x)$ satisfies the conditions (A1-A3), and assume $V(x)$ decays at least as $|x|^{-2}$ at infinity.
 Let $u_0 \in L^2(\mathbb{R}^n)$ be the initial condition, and $u(t) = e^{-iHt} u_0$.
 As we previously defined, let $Q_1 = Q_1(H)$ and $P_1 = P_1(H)$ be smoothed projections of $H$ on the intervals $[0,1]$ and $[1,\infty)$ respectively.
 Then for $1/2 < \sigma < 4(n-2)^2$, we have
 \begin{gather}
\label{eqn:app-low-energy}  \int \| \sqrt{-\Delta} \langle x \rangle^{-\sigma} Q_1 u(t) \|^2_{L_x^2} \ud t
+ \int \| \langle x \rangle^{-(\sigma+1)} Q_1 u(t) \|^2_{L_x^2} \ud t\leq C \|u_0\|^2_{L_x^2} , \\
\label{eqn:app-high-energy-1} \int \| \sqrt{-\Delta} \langle x \rangle^{-\sigma } H^{-1/4} P_1 u(t) \|^2_{L_x^2} \ud t
+ \int \| \langle x \rangle^{-(\sigma +1)} H^{-1/4} P_1 u(t) \|^2_{L_x^2} \ud t \leq C \|u_0\|^2_{L_x^2} , \\
\label{eqn:app-high-energy-2} \int \| \sqrt{H} \langle x \rangle^{-\sigma } H^{-1/4} P_1 u(t) \|^2_{L_x^2} \ud t \leq C \|u_0\|^2_{L_x^2} , \\
\label{eqn:app-high-energy-3} \int \| \langle x \rangle^{-\sigma} H^{1/4} P_1 u(t) \|^2_{L_x^2} \ud t \leq C \|u_0\|^2_{L_x^2} .
 \end{gather}
 Especially, we have
 \begin{equation} \label{eqn:app}
  \int \| \langle x \rangle^{-(\sigma +1)} u(t) \|^2_{L_x^2} \ud t \leq C \|u_0\|^2_{L_x^2}.
 \end{equation}
\end{thm}

\begin{proof}
For equation (\ref{eqn:app-low-energy}), using the conservation law of energy and the fact that $Q_1 \gamma_N Q_1$ is a bounded operator, we get:
\begin{equation}
\begin{split}
& \int_0^T \frac{\partial}{\partial t} (u(t), Q_1 \gamma_N Q_1 u(t)) \ud t \\
=& (u(T), Q_1 \gamma_N Q_1 u(T)) - (u(0), Q_1 \gamma_N Q_1 u(0)) \\
\leq& C \| u_0 \|_{L_x^2}^2.
\end{split}
\end{equation}
On the other hand, applying Theorem (\ref{thm:general}), we have:
\begin{equation}
\begin{split}
& \int_0^T \frac{\partial}{\partial t} (u(t), Q_1 \gamma_N Q_1 u(t)) \ud t \\
=& \int_0^T (u(t), i[H, Q_1 \gamma_N Q_1] u(t)) \ud t \\
=& \int_0^T (u(t), iQ_1[H, \gamma_N]Q_1 u(t)) \ud t \\
\geq& \int_0^T (u(t), Q_1 \widetilde{C} \sum_{k=-N}^{N} \langle x-k \rangle^{-\sigma} (-\Delta) \langle x-k \rangle^{-\sigma} Q_1 u(t)) \ud t \\
\geq& \widetilde{C} \int_0^T \| \sqrt{-\Delta} \langle x \rangle^{-\sigma} Q_1 u(t) \|^2_{L_x^2} \ud t, \text{~with~} \widetilde{C} >0.
\end{split}
\end{equation}
So for any $T$,
\begin{equation}
\int_0^T \| \sqrt{-\Delta} \langle x \rangle^{-\sigma} Q_1 u(t) \|^2_{L_x^2} \ud t \leq C \| u_0 \|_{L_x^2}^2,
\end{equation}
which means that
\begin{equation}
\int \| \sqrt{-\Delta} \langle x \rangle^{-\sigma} Q_1 u(t) \|^2_{L_x^2} \ud t \leq C \| u_0 \|_{L_x^2}^2.
\end{equation}
Then the Hardy's Inequality implies
\begin{equation}
\int \| \langle x \rangle^{-(\sigma + 1)} Q_1 u(t) \|^2_{L_x^2} \ud t \leq \int \| \sqrt{-\Delta} \langle x \rangle^{-\sigma} Q_1 u(t) \|^2_{L_x^2} \ud t \leq C \| u_0 \|_{L_x^2}^2.
\end{equation}
Hence we proved the equation (\ref{eqn:app-low-energy}).

Similarly, one can prove equation (\ref{eqn:app-high-energy-1}) by using Theorem (\ref{thm:general}),
conservation law and the fact that $P_1 H^{-1/4} \gamma_N H^{-1/4} P_1$ is bounded operator.

Equation (\ref{eqn:app-high-energy-2}) is a direct consequence of equation (\ref{eqn:app-high-energy-1}).

To prove equation (\ref{eqn:app-high-energy-3}), we need to commute $\sqrt{H}$ and $\langle x \rangle^{-\sigma}$, and estimate the correction terms by commutator expansion lemma.

Combining equation (\ref{eqn:app-low-energy}) and equation (\ref{eqn:app-high-energy-3}), one reaches the estimate (\ref{eqn:app}).
\end{proof}

With the help of Theorem \ref{thm:decay}, we prove the Strichartz estimates for $H = -\Delta +V(x)$.

\begin{thm}\label{thm:Strichartz}
 For dimension $n \geq 3$, suppose $H = -\Delta +V(x)$, with $V(x)$ axially repulsive, i.e. $V(x)$ satisfies the conditions (A1-A3).
 And suppose $V(x) = O(|x|^{-(2+\sigma_0)})$, for some $\sigma_0 >1/2$.
 Let $u_0 \in L^2(\mathbb{R}^n)$ be the initial condition, and $u(t) = e^{-iHt} u_0$. 
 Then we have the homogeneous Strichartz estimate
 \begin{equation}
  \| e^{-itH} u_0 \|_{L_t^q L_x^r (\mathbb{R}\times\mathbb{R}^n)} \leq C \|u_0\|_{L_x^2(\mathbb{R}^n)},
 \end{equation}
the dual homogeneous Strichartz estimate
\begin{equation}
 \| \int_{\mathbb{R}} e^{isH} F(s) \ud s \|_{L_x^2(\mathbb{R}^n)} \leq C \| F \|_{L_t^{\tilde{q}'} L_x^{\tilde{r}'} (\mathbb{R}\times\mathbb{R}^n)},
\end{equation}
and the inhomogeneous Strichartz estimate
\begin{equation}
 \| \int_{s <t} e^{-i(t-s)H} F(s) \ud s\|_{L_t^q L_x^r(\mathbb{R}^n)} \leq C \| F \|_{L_t^{\tilde{q}'} L_x^{\tilde{r}'} (\mathbb{R} \times \mathbb{R}^n)},
\end{equation}
with the pairs $(q,r)$ and $(\tilde{q},\tilde{r})$ are admissible exponents: $2\leq q, r \leq \infty$, $\frac{2}{q} + \frac{n}{r} = \frac{n}{2}$.
However, the endpoint case ($q = \tilde{q} =2$) of inhomogeneous Strichartz estimate is only available for $n=3$.
\end{thm}

\begin{proof}
 Use Duhamel formula and endpoint Strichartz estimates for free Schr\"{o}dinger operator, we have
 \begin{equation}
  \begin{split}
   \| u(t) \|_{L_t^2 L_x^{2n/(n-2)}} &\leq C \| u_0 \|_{L_x^2} + \| \int_0^t e^{i\Delta (t-s)} V(x) u(s) \ud s \|_{L_t^2 L_x^{2n/(n-2)}} \\
   &\leq C \| u_0 \|_{L_x^2} + \| V(x) u(t) \|_{L_t^2 L_x^{2n/(n+2)}}.
  \end{split}
 \end{equation}

 Take $\sigma \in (1/2, \sigma_0)$, then
 Theorem (\ref{thm:decay}) says $\| \langle x \rangle^{-(\sigma+1)} u(t) \|_{L_t^2 L_x^2} \leq C \| u_0 \|_{L_x^2}$,
 and we also have $\| V(x)\langle x \rangle^{\sigma+1} \|_{L_x^n}$ is bounded, so
 \begin{equation}
 \begin{split}
  \| V(x) u(t) \|_{L_t^2 L_x^{2n/(n+2)}} &= \| V(x) \langle x \rangle^{\sigma+1} \langle x \rangle^{-(\sigma+1)} u(t) \|_{L_t^2 L_x^{2n/(n+2)}} \\
  &\leq \| V(x) \langle x \rangle^{\sigma+1} \|_{L_x^n} \cdot \| \langle x \rangle^{-(\sigma+1)} u(t) \|_{L_t^2 L_x^2} \\
  &\leq C \| u_0 \|_{L_x^2}.
 \end{split}
  \end{equation}
 This completes the proof of homogeneous Strichartz estimate for $(q,r) = (2, 2n/(n-2))$, and the other endpoint is trivial.
 So by interpolation, we proved the homogeneous Strichartz estimate.

 Then by duality, we have the dual homogeneous Strichartz estimate, 
 which leads to the non-endpoint inhomogeneous Strichartz estimate under the help of Christ-Kiselev lemma.
 For the endpoint inhomogeneous Strichartz estimate in dimension 3, one can use \cite{beceanu2012schrodinger} to get dispersive estimate, 
 which implies the endpoint case by \cite{keel1998endpoint}.
\end{proof}

The Strichartz estimate for $H$ provides us the key to the global $H_x^1$ well-posedness of (e.g.) cubic defocusing nonlinear Schr\"odinger equation with axially repulsive potential
\begin{equation} \label{eqn:NLS}
 \left\{
 \begin{aligned}
  & i\partial_t u(t) = H u(t) + \lambda |u(t)|^{p-1} u(t) , \\
  & u(0) \in H_x^1 (\mathbb{R}^n).
   \end{aligned}
 \right.
\end{equation}
If $V(x)$ satisfies (A4) also, 
then we can prove the Morawetz and interaction Morawetz estimates using Theorem \ref{thm:interaction-morawetz},
which leads to the scattering of the solution to equation (\ref{eqn:NLS}).
By Theorem \ref{thm:interaction-morawetz},
$$i[-\Delta +V, \gamma_N + \sum_{c \in \sym\{x'\}} \gamma_{c}^{Mor}] 
\geq i[-\Delta, \sum_{c \in \sym\{x'\}} \gamma_{c}^{Mor}] + (1-\epsilon) i[-\Delta,\gamma_N],$$ 
the multiplier $\gamma_N + \sum_{c \in \sym\{x'\}} \gamma_{c}^{Mor}$ adapted to $H$, 
will replace the role of $\gamma_{x'}^{Mor}$ in proving (interaction) Morawetz estimates.
Then the standard Morawetz and interaction Morawetz estimates follow,
since the remainder terms, coming from the potential, are absorbed by the terms coming from $i[-\Delta, \gamma_N]$.

\subsection{Time Dependent Potential}
If $H = -\Delta + V(x,t)$, then even if $V$ is repulsive for each $t$, one cannot use the standard Morawetz estimates.
The reason is we have to replace $\gamma$ by $\gamma_{c(t)}$, $c(t)$ is the center of $V(x,t)$.
However, with the multi-center vector fields method, one can use the same $\gamma_N$ adapted to $V(x,t)$ for all $t$, 
if $V(x,t)$ remains axially repulsive and satisfies certain uniformity conditions as below:
\begin{enumerate}
 \item[(B1)] $V(x,t)$ satisfies the axially repulsive conditions (A1-A4), relative to the same axis.
 \item[(B2)] $V(x,t)$ satisfies the axially repulsive conditions (A1-A4),
 the constants $L$ and $\Lambda_{\delta}$ remain uniform in $t$.
\end{enumerate}
If the potential $V(x,t)$ satisfies the conditions (B1) and (B2), 
then we have the same estimates as in Theorem \ref{thm:general} and Theorem \ref{thm:interaction-morawetz} for the time dependent potential,
and we can prove interaction Morawetz estimate for such time dependent potentials.

\begin{example}
 Suppose $V_j(x_1,\vec{y})$ satisfies the conditions (A1-A4), for $j = 1, 2, \dots, M$.
 Then define the time dependent potential:
 \begin{equation}
V(x,t) \equiv \sum_1^M V_j(x_1 - \beta_j(t),\vec{y})
 \end{equation}
satisfies the conditions (B1) and (B2), if
 \begin{equation}
  \max_j \sup_t |\beta_j(t)| < \beta_0 < \infty.
 \end{equation}
\end{example}

\begin{example}
 Suppose $V(x_1,\vec{y})$ satisfies the conditions (A1-A4). Then define the time dependent potential:
 \begin{equation}
  V(x,t) \equiv V(x_1, \lambda(t) \vec{y})
 \end{equation}
satisfies the conditions (B1) and (B2), if
\begin{equation}
 0<\lambda_0 < \inf_t \lambda(t) \leq \sup_t \lambda(t) < \lambda_{\infty} < \infty.
\end{equation}
\end{example}

\begin{example}
 Suppose $V_j(x_1,\vec{y})$ satisfies the conditions (A1-A4), for $j = 1, 2, \dots, M$.
 Then define the time dependent potential:
 \begin{equation}
  V(x,t) \equiv \sum_1^M V_j(x_1 - \beta_j(t), \lambda_j(t) \vec{y})
 \end{equation}
satisfies the conditions (B1) and (B2), if
\begin{gather}
  \max_j \sup_t |\beta_j(t)| < \beta_0 < \infty, \\
  0<\lambda_0 < \min_j \inf_t \lambda_j(t) \leq \max_j \sup_t \lambda_j(t) < \lambda_{\infty} < \infty.
\end{gather}
\end{example}

\section{General Compactly Supported Potential}
In the general case, we assume $V(x) \geq 0$ is compactly supported and sufficiently smooth, and $H = -\Delta +V(x)$.
With no additional assumption to the potential function, we can not find $\gamma_N$ to shrink the region where $i[V(x),\gamma_N]$ is negative as we did in previous situations.
In general one uses the positive commutator methods based on the Mourre estimate in this case, which applies at localized energies away from thresholds ($0,\infty$, in our case).
See e.g.\cite{sigal1988local,hunziker2000quantum,soffer2011monotonic,amrein1996commutator,donninger2012pointwise} and cited references.
However, we can still get positive commutator in the high energy case and close to zero energy case (in higher dimensions).

\subsection{High Energy}

\begin{thm}\label{thm:high-energy}
Suppose $V(x) \geq 0$ is compactly supported and smooth, and $H = -\Delta + V(x)$.
Let $\gamma = i[-\Delta, F(x)]$, and $1/2 < \sigma < 4(n-2)^2$;
then there exists $K_0$ such that $P_K(H) i[H,\gamma] P_K(H)$ are positive operators for all $K \geq K_0$.
To be precise, for any $0< \epsilon < 4 - \sigma/(n-2)^2$, there exists $K_{\epsilon}$, such that for any $K\geq K_{\epsilon}$,
\begin{equation}
P_K (H) i[H,\gamma] P_K(H) \geq (4-\frac{\sigma}{(n-2)^2} -\epsilon) P_K(H) \frac{1}{\langle x \rangle^{\sigma}} H \frac{1}{\langle x \rangle^{\sigma}} P_K(H).
\end{equation}
\end{thm}

\begin{proof}
We already know that
\begin{equation}
\begin{split}
i[H,\gamma] \geq & -(4-\frac{\sigma}{(n-2)^2}) \frac{1}{\langle x \rangle^{\sigma}} \Delta \frac{1}{\langle x \rangle^{\sigma}} -2 \nabla F \cdot \nabla V \\
= & (4-\frac{\sigma}{(n-2)^2}) \frac{1}{\langle x \rangle^{\sigma}} H \frac{1}{\langle x \rangle^{\sigma}}
-(4-\frac{\sigma}{(n-2)^2}) \frac{1}{\langle x \rangle^{\sigma}} V(x) \frac{1}{\langle x \rangle^{\sigma}} -2 \nabla F \cdot \nabla V.
\end{split}
\end{equation}
Let $W(x) = (4-\frac{\sigma}{(n-2)^2}) V(x) + 2 \langle x \rangle^{2\sigma} \nabla F \cdot \nabla V$, so we only have to prove
\begin{equation}
\epsilon P_K(H) \frac{1}{\langle x \rangle^{\sigma}} H \frac{1}{\langle x \rangle^{\sigma}} P_K(H)
\geq P_K(H) \frac{1}{\langle x \rangle^{\sigma}} W(x) \frac{1}{\langle x \rangle^{\sigma}} P_K(H),
\end{equation}
for all $K$ greater than some constant $K_{\epsilon}$. We will need the following proposition to estimate the left hand side.

\begin{prop} \label{prop:high-energy}
Under the same assumptions as in Theorem (\ref{thm:high-energy}).
For $K$ large enough, we have the following estimate
\begin{equation}
P_K(H) \frac{1}{\langle x \rangle^{\sigma}} H \frac{1}{\langle x \rangle^{\sigma}} P_K(H)
\geq C K P_K(H) \frac{1}{\langle x \rangle^{\sigma}} \frac{1}{\langle x \rangle^{\sigma}} P_K(H),
\end{equation}
for some constant $C>0$.
\end{prop}

Theorem (\ref{thm:high-energy}) follows easily from the Proposition (\ref{prop:high-energy}), since $W(x)$ is a bounded function.
\end{proof}

To prove Proposition (\ref{prop:high-energy}), the idea is to commute $P_K(H)$ through $1/\langle x \rangle^{\sigma}$ and use the fact that $P_K(H) H P_K(H) \geq K P_K^2(H)$. After that, we need to commute $P_K(H)$ through $1/\langle x \rangle^{\sigma}$ again to the outside. So we have to control the error terms come from commuting $P_K(H)$ and $1/\langle x \rangle^{\sigma}$.

Using commutator expansion lemma, we get
\begin{equation}
\begin{split}
&[\langle x \rangle^{\sigma}, P_K(H)] \\
=& P'_K(H) \adj_{H}^{(1)} (\langle x \rangle^{\sigma})
+ \frac{1}{2!} P''_K(H) \adj_{H}^{(2)}(\langle x \rangle^{\sigma}) + \cdots
+ \frac{1}{m!} P^{(m)}_K(H) \adj_{H}^{(m)}(\langle x \rangle^{\sigma}) + R_{m+1},
\end{split}
\end{equation}
or
\begin{equation}
\begin{split}
& [\langle x \rangle^{\sigma}, P_K(H)] \\
=& \adj_{H}^{(1)}(\langle x \rangle^{\sigma}) P'_K(H)
+ \frac{1}{2!} \adj_{H}^{(2)}(\langle x \rangle^{\sigma}) P''_K(H) + \cdots
+ \frac{1}{m!} \adj_{H}^{(m)}(\langle x \rangle^{\sigma}) P^{(m)}_K(H) + \widetilde{R}_{m+1}.
\end{split}
\end{equation}
Here the remainder term $R_m$ is given by:
\begin{equation}
R_m = i^m \int_{-\infty}^{\infty} \widehat{P_K}(s_1) e^{is_1H} \ud s_1
\int_0^{s_1} \ud s_2 \int_0^{s_2} \cdots \ud s_m \int_0^{s_m} e^{-ipH} \adj_{H}^{(m)}(\langle x \rangle^{\sigma}) e^{ipH} \ud p.
\end{equation}
We have similar expression for $\widetilde{R}_m$ which we omit here.

\begin{prop} \label{commutation}
Let $m \geq \max \{4,\sigma\}$, then the following estimates hold:
\begin{enumerate}
\item $\adj_{H}^{(j)}(\langle x \rangle^{\sigma}) P^{(j)}_K(H) = \langle x \rangle^{\sigma - j} \frac{1}{\sqrt{K}^j} O(1)$, for $j=1,2, \cdots,m-1$.
\item $ (H+1) R_m =  \frac{1}{K} O(1)$.
\item $[\langle x \rangle^{\sigma}, P_K(H)] = \langle x \rangle^{\sigma-1} \frac{1}{\sqrt{K}} O(1)$, $H [\langle x \rangle^{\sigma}, P_K(H)] = \langle x \rangle^{\sigma - 1} \sqrt{K} O(1)$.
\item $[\langle x \rangle^{-\sigma}, P_K(H)] = \langle x \rangle^{-\sigma-1} \frac{1}{\sqrt{K}} O(1)$, $H [\langle x \rangle^{-\sigma}, P_K(H)] = \langle x \rangle^{-\sigma - 1} \sqrt{K} O(1)$.
\item $H \langle x \rangle^{-\sigma} [\langle x \rangle^{\sigma}, P_K(H)] = \sqrt{K} O(1)$.
\end{enumerate}
\end{prop}

\begin{proof}
For the first estimate, we compute explicitly:
\begin{equation}\begin{split}
[\langle x \rangle^{\sigma}, H] &= [\langle x \rangle^{\sigma}, -\Delta] \\
&= a\sigma \frac{n + a(\sigma + n-2 )r^2}{1+ar^2} \langle x \rangle^{\sigma-2} + 2a\sigma \langle x \rangle^{\sigma-1} \frac{x_j}{\langle x \rangle} \partial_j.
\end{split}\end{equation}
Since $\partial_j$ are bounded by $\sqrt{H}$, $\supp P_K'$ is contained in the interval $[K, 3K]$ and $\| P_K'(H) \| \leq \frac{1}{K}$, we get the first estimate for $j=1$.

We can use similar argument to prove the first estimate for $j=2,\cdots, m-1$.

For the remainder term $R_m$, we write
\begin{equation}
 e^{isH} = \frac{1}{i(H + 1)} (\frac{\partial}{\partial s} + i) e^{isH} = \frac{1}{(i(H + 1))^{m-1}} (\frac{\partial}{\partial s} + i)^{m-1} e^{isH}.
\end{equation}
Use integration by parts, we then get
\begin{equation}\begin{split}
& \|  (H+1) R_m \| \\
= & \| \int_{-\infty}^{\infty} \widehat{P_K}(s_1) \frac{1}{(H + 1)^{m-2}} (\frac{\partial}{\partial s_1} + i)^{m-1} e^{is_1H} \ud s_1
\int_0^{s_1} \cdots \ud s_m \int_0^{s_m} e^{-ipH} \adj_{H}^{(m)}(\langle x \rangle^{\sigma}) e^{ipH} \ud p \| \\
\leq & C \| \frac{1}{(H+1)^{m-2}} \adj_{H}^{(m)}(\langle x \rangle^{\sigma})\| \sum_{j=0}^{m-1} \int |\widehat{x^j P^{(j+1)}_K}(s)| \ud s \\
= & \frac{C}{K} \| \frac{1}{(H+1)^{m-2}} \adj_{H}^{(m)}(\langle x \rangle^{\sigma})\| \sum_{j=0}^{m-1} \int |\widehat{x^j P^{(j+1)}}(s)| \ud s.
\end{split}\end{equation}
Here we used an identity
\begin{equation}
\int |\widehat{x^j P^{(j+1)}_K}(s)| \ud s = \int |\widehat{x^j P^{(j+1)}}(Ks)| \ud s = \frac{1}{K} \int |\widehat{x^j P^{(j+1)}}(s)| \ud s,
\end{equation}
and the fact that $\frac{1}{(H+1)^{m-2}} \adj_{H}^{(m)}(\langle x \rangle^{\sigma})$ is bounded.
Thus we have proved the second estimate.

The third estimate is the direct consequence of the first and second estimates, using commutator expansion lemma.

Using exactly the same method, we get the fourth estimate.

For the last one, we have
\begin{equation}\begin{split}
&H \frac{1}{\langle x \rangle^{\sigma}} [\langle x \rangle^{\sigma}, P_K(H)] \\
=& H [\langle x \rangle^{\sigma}, P_K(H)] \frac{1}{\langle x \rangle^{\sigma}} + H [ \frac{1}{\langle x \rangle^{\sigma}}, [\langle x \rangle^{\sigma}, P_K(H)]] \\
=& H [\langle x \rangle^{\sigma}, P_K(H)] \frac{1}{\langle x \rangle^{\sigma}} +H [P_K(H), \langle x \rangle^{\sigma}] \frac{1}{\langle x \rangle^{\sigma}} + H [P_K(H), \frac{1}{\langle x \rangle^{\sigma}}] \langle x \rangle^{\sigma}\\
=& \sqrt{K} O(1).
\end{split}\end{equation}
\end{proof}

\begin{proof}[Proof of Proposition (\ref{prop:high-energy})]
Since $P_K(x) \equiv 1$ on $[3K, \infty)$, we have $P_{3K}(x) = P_{3K}(x) P_K(x)$. Then we get
\begin{equation}
P_{3K}(H) \frac{1}{\langle x \rangle^{\sigma}} H \frac{1}{\langle x \rangle^{\sigma}} P_{3K}(H) = P_{3K}(H) P_K(H) \frac{1}{\langle x \rangle^{\sigma}} H \frac{1}{\langle x \rangle^{\sigma}} P_K(H) P_{3K}(H)
\end{equation}

Now commute $P_K(H)$ through $\langle x \rangle^{-\sigma}$:
\begin{equation}
\begin{split}
& P_K(H) \frac{1}{\langle x \rangle^{\sigma}} H \frac{1}{\langle x \rangle^{\sigma}} P_K(H) \\
= & \frac{1}{\langle x \rangle^{\sigma}} P_K(H) H P_K(H) \frac{1}{\langle x \rangle^{\sigma}} \\
& + [P_K(H), \frac{1}{\langle x \rangle^{\sigma}}] H \frac{1}{\langle x \rangle^{\sigma}} P_K(H) + \frac{1}{\langle x \rangle^{\sigma}} P_K(H) H [\frac{1}{\langle x \rangle^{\sigma}}, P_K(H)]
\end{split}
\end{equation}
We then use the fact that $P_K(H) H P_K(H) \geq K P_K^2(H)$ and get
\begin{equation}
\begin{split}
& P_K(H) \frac{1}{\langle x \rangle^{\sigma}} H \frac{1}{\langle x \rangle^{\sigma}} P_K(H) \\
\geq & \frac{1}{\langle x \rangle^{\sigma}} P_K(H) K P_K(H) \frac{1}{\langle x \rangle^{\sigma}} \\
& - \frac{1}{\langle x \rangle^{\sigma}} [P_K(H), \langle x \rangle^{\sigma}]\frac{1}{\langle x \rangle^{\sigma}} H \frac{1}{\langle x \rangle^{\sigma}} P_K(H) \\
& - \frac{1}{\langle x \rangle^{\sigma}} P_K(H) H \frac{1}{\langle x \rangle^{\sigma}}[\langle x \rangle^{\sigma}, P_K(H)]\frac{1}{\langle x \rangle^{\sigma}} \\
=& K P_K(H) \frac{1}{\langle x \rangle^{\sigma}} \frac{1}{\langle x \rangle^{\sigma}} P_K(H) \\
& + K \frac{1}{\langle x \rangle^{\sigma}} [P_K(H), \langle x \rangle^{\sigma}]\frac{1}{\langle x \rangle^{\sigma}} P_K(H) \frac{1}{\langle x \rangle^{\sigma}} \\
& + K P_K(H) \frac{1}{\langle x \rangle^{\sigma}} \frac{1}{\langle x \rangle^{\sigma}}[\langle x \rangle^{\sigma}, P_K(H)]\frac{1}{\langle x \rangle^{\sigma}} \\
& - \frac{1}{\langle x \rangle^{\sigma}} [P_K(H), \langle x \rangle^{\sigma}]\frac{1}{\langle x \rangle^{\sigma}} H \frac{1}{\langle x \rangle^{\sigma}} P_K(H) \\
& - \frac{1}{\langle x \rangle^{\sigma}} P_K(H) H \frac{1}{\langle x \rangle^{\sigma}}[\langle x \rangle^{\sigma}, P_K(H)]\frac{1}{\langle x \rangle^{\sigma}}.
\end{split}
\end{equation}

Sandwich the above inequality by $P_{3K}(H)$, and use the estimates in Proposition (\ref{commutation}), we get
\begin{equation}
\begin{split}
& P_{3K}(H) \frac{1}{\langle x \rangle^{\sigma}} H \frac{1}{\langle x \rangle^{\sigma}} P_{3K}(H) \\
\geq & K P_{3K}(H) \frac{1}{\langle x \rangle^{\sigma}} \frac{1}{\langle x \rangle^{\sigma}} P_{3K}(H)
+ \sqrt{K} P_{3K}(H) \frac{1}{\langle x \rangle^{\sigma}} O(1) \frac{1}{\langle x \rangle^{\sigma}} P_{3K}(H).
\end{split}
\end{equation}

We can choose $K$ sufficiently large, then the error term can be controlled by the main term. Thus we proved the proposition.
\end{proof}

\subsection{Low Energy}
For this case, we only consider the special case that $n\geq 5$ and $\sigma=1$.
\begin{thm}\label{thm:low-energy}
Suppose $V(x) \geq 0$ is compactly supported and $C^2$, and $H = -\Delta + V(x)$.
We define $\gamma = i[-\Delta, F(x)]$. Then there exists $\xi_0$ such that
$Q_{\xi}(H) i[H,\gamma] Q_{\xi}(H)$ are positive operators for all $0 < \xi \leq \xi_0$.
To be precise, for any $\epsilon \in (0, 1)$, there exists $\xi_{\epsilon}$, such that for any $\xi \in (0, \xi_{\epsilon}]$,
\begin{equation}
Q_{\xi} (H) i[H,\gamma] Q_{\xi} (H) \geq (1 -\epsilon) Q_{\xi}(H) i[-\Delta, \gamma] Q_{\xi}(H).
\end{equation}
Here $Q_{\xi}(\lambda) = Q(\lambda/\xi)$, with $Q(\lambda)$ is a fixed smoothed characteristic function of the interval $[0,1]$.
\end{thm}

\begin{proof}
We already know that
\begin{equation}
\begin{split}
i[H,\gamma] = & (1-\epsilon)i[-\Delta, \gamma] + i[-\epsilon\Delta, \gamma] + i[V,\gamma]\\
\geq & (1-\epsilon)i[-\Delta, \gamma] + \{ -(4-\frac{\sigma}{(n-2)^2}) \epsilon \frac{1}{\langle x \rangle^{\sigma}} \Delta \frac{1}{\langle x \rangle^{\sigma}} -2 \nabla F \cdot \nabla V \} \\
\geq & (1-\epsilon)i[-\Delta, \gamma] 
+ \{ (4-\frac{\sigma}{(n-2)^2}) \epsilon \frac{(n-2)^2}{4} \frac{1}{\langle x \rangle^{\sigma}} \frac{a}{\langle x \rangle^{2}} \frac{1}{\langle x \rangle^{\sigma}} 
-2 \nabla F \cdot \nabla V \}\\
= & (1-\epsilon)i[-\Delta, \gamma] 
+ \{ ((n-2)^2 - \frac{\sigma}{4}) a\epsilon \frac{1}{\langle x \rangle^{2}}  \frac{1}{\langle x \rangle^{2}} 
-2 \nabla F \cdot \nabla V\}.
\end{split}
\end{equation}

Let $U(x) = 2 \nabla F \cdot \nabla V$, $C = ((n-2)^2 - \frac{\sigma}{4}) a$, and sandwich the above inequality by $Q_{\xi}(H)$ (we use $Q_{\xi}$ for short):
\begin{equation}
\begin{split}
& Q_{\xi} (i[H,\gamma]) Q_{\xi} \\
\geq & (1 -\epsilon) Q_{\xi} i[-\Delta, \gamma] Q_{\xi}
+ \{ C\epsilon Q_{\xi}\frac{1}{\langle x \rangle^{2}}  \frac{1}{\langle x \rangle^{2}}Q_{\xi} 
- Q_{\xi} U(x) Q_{\xi} \} \\
= & (1 -\epsilon) Q_{\xi} i[-\Delta, \gamma] Q_{\xi} 
+ Q_{\xi}\frac{1}{\langle x \rangle^{2}} \Bigl\{ C\epsilon - \langle x \rangle^{2} Q_{3\xi} U(x) Q_{3\xi} \langle x \rangle^{2} \Bigr\} \frac{1}{\langle x \rangle^{2}}Q_{\xi}.
\end{split}
\end{equation}
Then we only have to prove there exists $\xi_{\epsilon} >0$,
\begin{equation}
\| \langle x \rangle^{2} Q_{3\xi} U(x) Q_{3\xi} \langle x \rangle^{2} \| \leq C \epsilon,
\end{equation}
for all $0 < \xi < \xi_{\epsilon}$. That means all we have to prove is
\begin{equation}
 \lim_{\xi \to 0} \| \langle x \rangle^{2} Q_{3\xi} U(x) Q_{3\xi} \langle x \rangle^{2} \| = 0.
\end{equation}
In fact, we prove the following stronger estimate:
\begin{lemma} \label{lemma:low-energy}
Let $\alpha_0 = \min \{(n/4-1), 1/2 \}$, for any $0 < \alpha < \alpha_0$,
\begin{equation}
\| \langle x \rangle^{2} Q_{3\xi} U(x) Q_{3\xi} \langle x \rangle^{2} \| = \xi^{2\alpha} O(1).
\end{equation}
\end{lemma}

The Theorem (\ref{thm:low-energy}) then follows from the Lemma (\ref{lemma:low-energy}).
\end{proof}

To prove the Lemma (\ref{lemma:low-energy}), we need to commute $\langle x \rangle^{2}$ with $Q_{3\xi}$:
\begin{equation}
\begin{split}
 & \langle x \rangle^{2} Q_{3\xi} U(x) Q_{3\xi} \langle x \rangle^{2} \\
 =& Q_{3\xi} \langle x \rangle^{2}  U(x)  \langle x \rangle^{2} Q_{3\xi} +
 [\langle x \rangle^{2}, Q_{3\xi}] U(x)  \langle x \rangle^{2} Q_{3\xi} \\
 & + Q_{3\xi} \langle x \rangle^{2}  U(x) [Q_{3\xi}, \langle x \rangle^{2}] +
 [\langle x \rangle^{2}, Q_{3\xi}] U(x) [Q_{3\xi}, \langle x \rangle^{2}].
\end{split}
\end{equation}
And all we have to prove is the following estimates:
\begin{prop}\label{prop:low-energy}
 Let $\alpha_0 = \min \{(n/4-1), 1/2 \}$, for any $0 < \alpha < \alpha_0$,
\begin{gather}
\label{eqn:low-energy-1} \| Q_{\xi}  \chi_{U}(x)\| = \xi^{1+\alpha} O(1), \\
\label{eqn:low-energy-2} \| [x_j, Q_{\xi}]  \chi_{U}(x)\| = \xi^{1/2 + \alpha} O(1), ~~j=1,\dots,n, \\
\label{eqn:low-energy-3} \| [\langle x \rangle^{2}, Q_{\xi}] \chi_{U}(x) \| = \xi^{\alpha} O(1).
\end{gather}
Here $\chi_{U}(x)$ is the smoothed characteristic function of $U(x)$ with compact support.
\end{prop}

\begin{proof}
For the first estimate, we write
\begin{equation}
 \| Q_{\xi}  \chi_{U}(x)\|
 = \| \Bigl( Q_{\xi} H^{1+\alpha} \Bigr) \Bigl(H^{-(1+\alpha)} \langle x \rangle^{-2(1+ \alpha)} \Bigr) \langle x \rangle^{2(1+ \alpha)} \chi_{U}(x)\| .
\end{equation}
We know from the weighted Hardy-Littlewood-Sobolev (WHLS) inequality that, for $0 \leq \kappa< \frac{n}{2}$, $|p|^{-\kappa} |x|^{-\kappa}$ is a bounded operator on $L^2(\mathbb{R}^n)$.
So if $H = -\Delta$, then equation (\ref{eqn:low-energy-1}) is a direct consequence of WHLS.
Now we have to prove $H^{-(1+\alpha)} \langle x \rangle^{-2(1+ \alpha)}$ is a bounded operator, for $H= -\Delta +V(x)$.
So equation (\ref{eqn:low-energy-1}) follows from the lemma:

\begin{lemma} \label{lemma:wHLS}
For dimension $n \geq 5$, let $\alpha_0 = \min \{(n/4-1), 1/2 \}$. Suppose that $H = -\Delta + V(x)$ is a self-adjoint operator,
with $|V(x)| \leq C \langle x \rangle^{-2(1+\alpha_0)}$ and $0 < C_1 (-\Delta) \leq H \leq C_2 (-\Delta)$.
Then
\begin{enumerate}
 \item $(H^{-1} - (-\Delta)^{-1}) \langle x \rangle^{-2}$, $H^{-1} \langle x \rangle^{-2}$ are bounded.
 \item For any $0< \alpha <\alpha_0$, we have $H^{-(1+\alpha)} \langle x \rangle^{-2(1+ \alpha)}$ is bounded.
 \end{enumerate}
\end{lemma}

\begin{proof}[Proof of Lemma (\ref{lemma:wHLS})]
Since $H \geq C_1 (-\Delta)$, we have
\begin{equation}
  \langle x \rangle^{-2} H^{-1} \langle x \rangle^{-2} \leq \frac{1}{C_1} \langle x \rangle^{-2} (-\Delta)^{-1} \langle x \rangle^{-2},
\end{equation}
in the sense of forms. By WHLS inequality, the RHS is bounded, then so is the LHS.
Use resolvent formula and WHLS inequality, the operator
\begin{equation}
\begin{split}
(H^{-1} - (-\Delta)^{-1}) \langle x \rangle^{-2} =& - (-\Delta)^{-1} V(x) H^{-1} \langle x \rangle^{-2}\\
=& - ((-\Delta)^{-1} \langle x \rangle^{-2}) (\langle x \rangle^{2} V(x) \langle x \rangle^{2}) (\langle x \rangle^{-2} H^{-1} \langle x \rangle^{-2})
\end{split}
\end{equation}
is bounded. Then $H^{-1} \langle x \rangle^{-2}$ is also bounded.

For $0<\alpha <\alpha_0 \leq 1/2$, we have
\begin{equation}
 H^{2\alpha} \geq C_{\alpha} (-\Delta)^{2\alpha},
 \end{equation}
which implies
\begin{equation}
 H^{-2\alpha} \leq \frac{1}{C_{\alpha}} (-\Delta)^{-2\alpha},
\end{equation}
in the form sense.
So we have
\begin{equation}
\| H^{-\alpha} \psi \| \leq \frac{1}{\sqrt{C_{\alpha}}} \| (-\Delta)^{-\alpha} \psi \|,
\end{equation}
for any $\psi$ such that the RHS is bounded. Then use this result and we get
\begin{equation}
 \begin{split}
  & \| H^{-(1+\alpha)} \langle x \rangle^{-2(1+\alpha)} \psi \| \\
  =& \| H^{-\alpha} \Bigl( (-\Delta)^{-1} - (-\Delta)^{-1} V(x) H^{-1} \Bigr) \langle x \rangle^{-2(1+\alpha)} \psi\| \\
  \leq & \frac{1}{\sqrt{C_{\alpha}}} \| (-\Delta)^{-\alpha} \Bigl( (-\Delta)^{-1} - (-\Delta)^{-1} V(x) H^{-1} \Bigr) \langle x \rangle^{-2(1+\alpha)} \psi\| \\
  =& \frac{1}{\sqrt{C_{\alpha}}} \| \Bigl( (-\Delta)^{-(1+ \alpha)} \langle x \rangle^{-2(1+\alpha)} - (-\Delta)^{-(1+ \alpha)} V(x) H^{-1} \langle x \rangle^{-2(1+\alpha)} \Bigr) \psi\| \\
  \leq& \widetilde{C_{\alpha}} \| \psi \|,
 \end{split}
\end{equation}
for any $\psi \in L^2_x$. So $H^{-(1+\alpha)} \langle x \rangle^{-2(1+\alpha)}$ is (can be extended to) a bounded operator.
\end{proof}

For equation (\ref{eqn:low-energy-2}), we write
\begin{equation}
\begin{split}
[x_j, Q_{\xi}] \chi_{U}(x) &= [x_j, Q_{\xi}Q_{3\xi}] \chi_{U}(x) \\
&= [x_j, Q_{\xi}] Q_{3\xi} \chi_{U}(x) + Q_{\xi}[x_j, Q_{3\xi}] \chi_{U}(x).
\end{split}
\end{equation}
Then we use commutator expansion lemma to analysis each term.
\begin{equation}
 \begin{split}
  [x_j, Q_{\xi}] Q_{3\xi} \chi_{U}(x) &= (Q'_{\xi} [x_j, H] + R_2) Q_{3\xi} \chi_{U}(x) \\
  &= ( 2i Q'_{\xi} p_j + R_2) Q_{3\xi} \chi_{U}(x) \\
  &= 2i Q'_{\xi} p_j Q_{3\xi} \chi_{U}(x) + R_2 Q_{3\xi} \chi_{U}(x).
 \end{split}
\end{equation}
\begin{equation}
\begin{split}
  Q_{\xi}[x_j, Q_{3\xi}] \chi_{U}(x) &= Q_{\xi} (Q'_{3\xi} [x_j, H] + \widetilde{R}_2) \chi_{U}(x) \\
  &= Q_{\xi} ( 2i p_j Q'_{3\xi} + \widetilde{R}_2) \chi_{U}(x) \\
  &= 2i Q_{\xi} p_j Q'_{3\xi} \chi_{U}(x) + Q_{\xi} \widetilde{R}_2 \chi_{U}(x).
\end{split}
\end{equation}

Here $p_j = -i \partial_j$ are the momentum operators, and $R_2$ is the remainder term in commutator expansion lemma
\begin{equation}
\begin{split}
 R_2 &= i^2 \int_{-\infty}^{\infty} \widehat{Q}_{\xi}(s) e^{isH} \ud s \int_0^s \int_0^{\mu} e^{-i\nu H} [[x_j,H],H] e^{i\nu H} \ud\nu \ud\mu \\
 &= - \int_{-\infty}^{\infty} \widehat{Q}_{\xi}(s) e^{isH} \ud s \int_0^s \int_0^{\mu} e^{-i\nu H} 2\frac{\partial V(x)}{\partial x_j} e^{i\nu H} \ud\nu \ud\mu.
\end{split}
\end{equation}
And $\widetilde{R}_2$ is similar to $R_2$:
\begin{equation}
 \begin{split}
 \widetilde{R}_2 &= -i^2\int_{-\infty}^{\infty} \widehat{Q}_{\xi}(s)  \int_0^s \int_0^{\mu} e^{i\nu H} [[x_j,H],H] e^{ - i\nu H} \ud\nu \ud\mu ~e^{isH} \ud s\\
 &= \int_{-\infty}^{\infty} \widehat{Q}_{\xi}(s) \int_0^s \int_0^{\mu} e^{i\nu H} 2\frac{\partial V(x)}{\partial x_j} e^{-i\nu H} \ud\nu \ud\mu ~e^{isH} \ud s.
\end{split}
\end{equation}

Use WHLS the same way we used for the first estimate, we get
\begin{equation}
\begin{split}
 \| 2i Q'_{\xi} p_j Q_{3\xi} \chi_{U}(x) \| =& \xi^{1/2+\alpha} O(1), \\
 \| 2i Q_{\xi} p_j Q'_{3\xi} \chi_{U}(x) \| =&  \xi^{1/2+\alpha} O(1).
\end{split}
\end{equation}
For the remainder terms, we need the following lemma:
\begin{lemma} \label{lemma:low-energy-remainder}
Suppose $W(x)$ is a bounded function of compact support, then
\begin{gather}
\label{eqn:remainder-1} \Bigl \| \Bigl( \int_{-\infty}^{\infty} \widehat{Q}_{\xi}(s) e^{isH} \ud s \int_0^s \int_0^{\mu} e^{-i\nu H} W(x) e^{i\nu H} \ud\nu \ud\mu \Bigr) Q_{3\xi} \chi_{U}(x) \Bigr \| = \xi^{1+2\alpha} O(1), \\
\label{eqn:remainder-2} \Bigl \| Q_{\xi} \Bigl( \int_{-\infty}^{\infty} \widehat{Q}_{3\xi}(s) \int_0^s \int_0^{\mu} e^{i\nu H} W(x) e^{-i\nu H} \ud\nu \ud\mu ~e^{isH} \ud s \Bigr) \chi_{U}(x) \Bigr \| = \xi^{1+\alpha} O(1).
\end{gather}
\end{lemma}

The equation (\ref{eqn:low-energy-2}) then follows directly from Lemma (\ref{lemma:low-energy-remainder}). We will prove the lemma later.

For equation (\ref{eqn:low-energy-3}), just use the same techniques as in proving equation (\ref{eqn:low-energy-2})
\begin{equation}
\begin{split}
[x_j^2, Q_{\xi}] \chi_{U}(x) &= [x_j^2, Q_{\xi}Q_{3\xi}] \chi_{U}(x) \\
&= [x_j^2, Q_{\xi}] Q_{3\xi} \chi_{U}(x) + Q_{\xi}[x_j^2, Q_{3\xi}] \chi_{U}(x).
\end{split}
\end{equation}
We compute
\begin{equation}\label{eqn:x_square_1}
 \begin{split}
  ~& [x_j^2, Q_{\xi}] Q_{3\xi} \chi_{U}(x)\\
  =~& (Q'_{\xi} [x_j^2, H] + R_2^{(2)}) Q_{3\xi} \chi_{U}(x) \\
  =~& 2i Q'_{\xi} (x_j p_j + p_j x_j) Q_{3\xi} \chi_{U}(x) + R_2^{(2)} Q_{3\xi} \chi_{U}(x)\\
  =~&  4i Q'_{\xi}  p_j x_j Q_{3\xi} \chi_{U}(x) -2Q'_{\xi} Q_{3\xi} \chi_{U}(x) + R_2^{(2)} Q_{3\xi} \chi_{U}(x)\\
  =~&  4i Q'_{\xi}  p_j Q_{3\xi} x_j \chi_{U}(x) + 4i Q'_{\xi}  p_j [x_j, Q_{3\xi}] \chi_{U}(x) -2Q'_{\xi} Q_{3\xi} \chi_{U}(x) + R_2^{(2)} Q_{3\xi} \chi_{U}(x),
 \end{split}
 \end{equation}
and
\begin{equation}\label{eqn:x_square_2}
\begin{split}
  ~&Q_{\xi}[x_j^2, Q_{3\xi}] \chi_{U}(x) \\
  =~& Q_{\xi}([x_j^2, H] Q'_{3\xi} + \widetilde{R}_2^{(2)} \chi_{U}(x) \\
  =~& 2i Q_{\xi} (x_j p_j + p_j x_j) Q'_{3\xi} \chi_{U}(x) + Q_{\xi} \widetilde{R}_2^{(2)} \chi_{U}(x)\\
  =~&  4i Q_{\xi}  p_j x_j Q'_{3\xi} \chi_{U}(x) -2Q_{\xi} Q'_{3\xi} \chi_{U}(x) + Q_{\xi} \widetilde{R}_2^{(2)} \chi_{U}(x)\\
  =~&  4i Q_{\xi}  p_j Q'_{3\xi} x_j \chi_{U}(x) + 4i Q_{\xi}  p_j [x_j, Q'_{3\xi}] \chi_{U}(x) -2Q_{\xi} Q'_{3\xi} \chi_{U}(x) + Q_{\xi} \widetilde{R}_2^{(2)} \chi_{U}(x).
\end{split}
\end{equation}

For the remainder terms,
\begin{gather}
 R_2^{(2)} = i^2 \int_{-\infty}^{\infty} \widehat{Q}_{\xi}(s) e^{isH} \ud s \int_0^s \int_0^{\mu} e^{-i\nu H} [[x_j^2,H],H] e^{i\nu H} \ud\nu \ud\mu, \\
 \widetilde{R}_2^{(2)} = -i^2 \int_{-\infty}^{\infty} \widehat{Q}_{\xi}(s) \int_0^s \int_0^{\mu} e^{i\nu H} [[x_j^2,H],H] e^{-i\nu H} \ud\nu \ud\mu ~e^{isH} \ud s.
\end{gather}
We can use the fact that
\begin{equation}
 [[x_j^2,H],H] = - 8p_j^2 + 4\frac{\partial V(x)}{\partial x_j} x_j.
\end{equation}
So we have
\begin{equation}
 [[x^2,H],H] = - 8p^2 + 4\nabla V(x) \cdot x = - 8H + 8V(x) + 4\nabla V(x) \cdot x.
\end{equation}

Then use WHLS inequality, the Lemma (\ref{lemma:low-energy-remainder}) and equation (\ref{eqn:low-energy-2}),
with the above computations, we prove every term in the RHS of (\ref{eqn:x_square_1}) and (\ref{eqn:x_square_2}) are of order $\xi^{\alpha}$.
So we proved (\ref{eqn:low-energy-3}).
\end{proof}

\begin{proof}[Proof of Lemma (\ref{lemma:low-energy-remainder}).]
We write $\widetilde{W} = H^{-1} W(x) H^{-(1+\alpha)}$ and $\widetilde {\chi} = H^{-(1+\alpha)} \chi_{U}(x)$.
They are both bounded operators due to the WHLS.
 \begin{equation}
  \begin{split}
   &  \Bigl( \int_{-\infty}^{\infty} \widehat{Q}_{\xi}(s) e^{isH} \ud s \int_0^s \int_0^{\mu} e^{-i\nu H} W(x) e^{i\nu H} \ud\nu \ud\mu \Bigr) Q_{3\xi} \chi_{U}(x) \\
   =& \Bigl( \int_{-\infty}^{\infty} \widehat{Q}_{\xi}(s) e^{isH} \ud s \int_0^s \int_0^{\mu} e^{-i\nu H} H \widetilde{W} e^{i\nu H} \ud\nu \ud\mu \Bigr) H^{2(1+\alpha)} Q_{3\xi} \widetilde {\chi} \\
   =& \Bigl(i \int_{-\infty}^{\infty} \widehat{Q}_{\xi}(s) e^{isH} \ud s \int_0^s \int_0^{\mu} \ud(e^{-i\nu H}) \widetilde{W} e^{i\nu H} \ud\mu \Bigr) H^{2(1+\alpha)} Q_{3\xi} \widetilde {\chi}  \\
   =& \Bigl(i \int_{-\infty}^{\infty} \widehat{Q}_{\xi}(s) e^{isH} \ud s \int_0^s (e^{-i\mu H} \widetilde{W} e^{i\mu H} - \widetilde{W}) \ud\mu \Bigr) H^{2(1+\alpha)} Q_{3\xi} \widetilde {\chi} \\
   & -\Bigl(i \int_{-\infty}^{\infty} \widehat{Q}_{\xi}(s) e^{isH} \ud s \int_0^s \int_0^{\mu} e^{-i\nu H} \widetilde{W} \ud(e^{i\nu H}) \ud\mu \Bigr) H^{2(1+\alpha)} Q_{3\xi} \widetilde {\chi}  \\
   =& \Bigl(i \int_{-\infty}^{\infty} \widehat{Q}_{\xi}(s) e^{isH} \ud s \int_0^s (e^{-i\mu H} \widetilde{W} e^{i\mu H} - \widetilde{W}) \ud\mu \Bigr) H^{2(1+\alpha)} Q_{3\xi} \widetilde {\chi} \\
  & + \Bigl( \int_{-\infty}^{\infty} \widehat{Q}_{\xi}(s) e^{isH} \ud s \int_0^s \int_0^{\mu} e^{-i\nu H} \widetilde{W} e^{i\nu H} \ud\nu \ud\mu \Bigr) H^{1+2(1+\alpha)} Q_{3\xi} \widetilde {\chi}.
  \end{split}
 \end{equation}

Since $\widetilde{W}$ is a bounded operator, then we have
\begin{equation}
 \begin{split}
  &\Bigl\| \int_{-\infty}^{\infty} \widehat{Q}_{\xi}(s) e^{isH} \ud s \int_0^s \int_0^{\mu} e^{-i\nu H} \widetilde{W} e^{i\nu H} \ud\nu \ud\mu \Bigr\|\\
 \leq& \|\widetilde{W}\| \Bigl\| \int_{-\infty}^{\infty} \widehat{Q}_{\xi}(s) e^{isH} \ud s \int_0^s \int_0^{\mu} \ud\nu \ud\mu \Bigr\|
 = \xi^{-2} O(1),
\end{split}
\end{equation}
\begin{equation}
 \begin{split}
 &\Bigl\| \int_{-\infty}^{\infty} \widehat{Q}_{\xi}(s) e^{isH} \ud s \int_0^s (e^{-i\mu H} \widetilde{W} e^{i\mu H} - \widetilde{W}) \ud\mu \Bigr\| \\
 \leq& \|\widetilde{W}\| \Bigl\| \int_{-\infty}^{\infty} \widehat{Q}_{\xi}(s) e^{isH} \ud s \int_0^s \ud\mu \Bigr\|
 = \xi^{-1} O(1).
 \end{split}
\end{equation}
We also know $\widetilde{\chi}$ is bounded and that
\begin{gather}
\| H^{1+2(1+\alpha)} Q_{3\xi} \| \leq \xi^{3+2\alpha},\\
\| H^{2(1+\alpha)} Q_{3\xi} \| \leq \xi^{2+2\alpha}.
\end{gather}
Combining the above results, we proved equation (\ref{eqn:remainder-1}).

For equation (\ref{eqn:remainder-2}), we have to use WHLS to extract $H^2$ between the two compact functions $W(x)$ and $\chi_{U}(x)$.
Similarly we write $\widetilde{\widetilde{W}} = H^{-(1+\alpha)} W(x) H^{-1}$ and $\widetilde{\widetilde{\chi}} = H^{-1} \chi_{U}(x)$.
They are both bounded operators due to the WHLS.
\begin{equation}
\begin{split}
&  Q_{\xi} \Bigl( \int_{-\infty}^{\infty} \widehat{Q}_{3\xi}(s) e^{isH} \ud s \int_0^s \int_0^{\mu} e^{-i\nu H} W(x) e^{i\nu H} \ud\nu \ud\mu \Bigr) \chi_{U}(x) \\
=& Q_{\xi} H^{1+\alpha} \Bigl( \int_{-\infty}^{\infty} \widehat{Q}_{3\xi}(s) e^{isH} \ud s \int_0^s \int_0^{\mu} e^{-i\nu H} \widetilde{\widetilde{W}} e^{i\nu H} H^2 \ud\nu \ud\mu \Bigr) \widetilde{\widetilde{\chi}}. \\
\end{split}
\end{equation}
Then we can use integration by parts twice to move $H^2$ to the left of $\widetilde{\widetilde{W}}$ and hit $Q_{\xi}$,
with the same method as we prove equation (\ref{eqn:remainder-1}).
\begin{equation}
\begin{split}
& \int_0^s \int_0^{\mu} e^{-i\nu H} \widetilde{\widetilde{W}} e^{i\nu H} H^2 \ud\nu \ud\mu \\
=& -i\int_0^s \int_0^{\mu} e^{-i\nu H} \widetilde{\widetilde{W}} H \ud (e^{i\nu H})  \ud\mu \\
=& -i\int_0^s (e^{-i\mu H} \widetilde{\widetilde{W}} H e^{i\mu H} - \widetilde{\widetilde{W}} H) \ud\mu \\
& +i\int_0^s \int_0^{\mu} \ud (e^{-i\nu H}) \widetilde{\widetilde{W}} H e^{i\nu H}  \ud\mu \\
=& -(e^{-i\mu H} \widetilde{\widetilde{W}} e^{i\mu H})\Big|_0^{s} + \int_0^{s} \ud(e^{-i\mu H}) \widetilde{\widetilde{W}} e^{i\mu H}
+ i\widetilde{\widetilde{W}}H \int_0^s \ud \mu \\
& -iH \int_0^s \Big( e^{-i\nu H} \widetilde{\widetilde{W}} e^{i\nu H} \Big|_0^{\mu} - \int_0^{\mu} -iH  e^{-i\nu H} \widetilde{\widetilde{W}} e^{i\nu H} \ud \nu \Big) \ud \mu \\
=& (1 + sH + s^2 H^2) O(1) + is\widetilde{\widetilde{W}} H.
\end{split}
\end{equation}
So we have
\begin{equation}
\begin{split}
&  Q_{\xi} \Bigl( \int_{-\infty}^{\infty} \widehat{Q}_{3\xi}(s) e^{isH} \ud s \int_0^s \int_0^{\mu} e^{-i\nu H} W(x) e^{i\nu H} \ud\nu \ud\mu \Bigr) \chi_{U}(x) \\
=& Q_{\xi} H^{1+\alpha} \Bigl( \int_{-\infty}^{\infty} \widehat{Q}_{3\xi}(s) e^{isH} \ud s \int_0^s \int_0^{\mu} e^{-i\nu H} \widetilde{\widetilde{W}} e^{i\nu H} H^2 \ud\nu \ud\mu \Bigr) \widetilde{\widetilde{\chi}} \\
=& Q_{\xi} H^{1+\alpha} \Bigl( \int_{-\infty}^{\infty} \widehat{Q}_{3\xi}(s) e^{isH} \{ (1 + sH + s^2 H^2) O(1) + is\widetilde{\widetilde{W}} H \} \ud s \Bigr) \widetilde{\widetilde{\chi}} \\
=& \xi^{1 + \alpha} O(1) + Q_{\xi} H^{1+\alpha} \int_{-\infty}^{\infty} \widehat{Q}_{3\xi}(s) e^{isH} is\widetilde{\widetilde{W}} H \ud s \widetilde{\widetilde{\chi}} \\
=& \xi^{1 + \alpha} O(1) + \Bigl( \int_{-\infty}^{\infty} \widehat{Q}_{3\xi}(s) e^{isH} is \ud s \Bigr) Q_{\xi} H^{1+\alpha} \widetilde{\widetilde{W}} H  \widetilde{\widetilde{\chi}} \\
=& \xi^{1 + \alpha} O(1) + Q'_{3\xi} Q_{\xi} H^{1+\alpha} \widetilde{\widetilde{W}} H  \widetilde{\widetilde{\chi}} \\
=& \xi^{1 + \alpha} O(1).
\end{split}
\end{equation}
The last equality used the fact that $Q'_{3\xi} Q_{\xi}= 0$. Thus we proved equation (\ref{eqn:remainder-2}).
\end{proof}

\section{Examples}
\subsection{Potentials with Nondefinite Sign}
In this section, we discuss Schr\"odinger operators with potential function of special form.

\begin{thm}
 In dimension $n \geq 3$, suppose $H = -\Delta + V(x)$ is self-adjoint.
 For fixed $a>0$ and $1/2 < \sigma < 4 (n-2)^2$, if for some $0<\lambda<1$, $V(x)$ satisfies the following condition:
\begin{equation}
 \lambda (4 - \frac{\sigma}{(n-2)^2} ) \langle x \rangle^{-\sigma} (- \Delta) \langle x \rangle^{-\sigma}
 -2 f(r) (\frac{x}{r} \cdot \nabla V(x)) \geq 0,
\end{equation}
then $i[H, \gamma_0]$ is positive, and
\begin{equation}
 i[H, \gamma_0] \geq (1 - \lambda) ( 4 - \frac{\sigma}{(n-2)^2}) \langle x \rangle^{-\sigma} (-\Delta) \langle x \rangle^{-\sigma}.
\end{equation}
Especially, a slightly stronger but easier condition to verify implying that $i[H, \gamma_0]$ is positive is:
\begin{equation}
\lambda ((n-2)^2 - \sigma/4) \frac{1}{r^2 \langle x \rangle^{2\sigma}}
-2 \frac{M_{\sigma}}{\sqrt{a}} |\nabla V(x)| \geq 0.
\end{equation}
\end{thm}

\begin{proof}

Use our previous result,
\begin{equation}
 \begin{split}
i[H,\gamma_0] \geq& -(4 - \frac{\sigma}{(n-2)^2})  \langle x \rangle^{-\sigma} \Delta  \langle x \rangle^{-\sigma} - 2 \nabla F \cdot \nabla V \\
=& -(1-\lambda)(4 - \frac{\sigma}{(n-2)^2})  \langle x \rangle^{-\sigma} \Delta  \langle x \rangle^{-\sigma} \\
&+ \{-\lambda (4 - \frac{\sigma}{(n-2)^2}) \langle x \rangle^{-\sigma} \Delta \langle x \rangle^{-\sigma}
 -2 f(r) (\frac{x}{r} \cdot \nabla V)\} \\
\geq&  -(1-\lambda)(4 - \frac{\sigma}{(n-2)^2})  \langle x \rangle^{-\sigma} \Delta  \langle x \rangle^{-\sigma}.
 \end{split}
\end{equation}
So if $0<\lambda < 1$, then $i[H, \gamma_0]$ will be a positive operator.
\end{proof}

\begin{example}
We consider potential function of this form:
\begin{equation}
 V(x) = V(r) = \frac{-1}{b + c r^{2+\epsilon}}.
\end{equation}
Then a sufficient condition for $i[H, \gamma]$ to be positive is:
\begin{equation}\label{example:negative}
\begin{split}
 1/2 < \sigma \leq 1, ~~~~\lambda<4-\sigma/(n-2)^2, ~~~~ \epsilon = 2\sigma -1, \\
 b \geq \frac{4(2+\epsilon)M_{\sigma} }{a (n-2)^2 \lambda}, ~~~~ c\geq \frac{16(2+\epsilon) M_{\sigma} a^{\epsilon/2}}{(n-2)^2\lambda}.
\end{split}
\end{equation}
If we choose the parameters $b$ and $c$ in (\ref{example:negative}) with the equality satisfied, then
\begin{equation}
 V(x) = V(r) = \frac{-a(n-2)^2 \lambda}{M_{\sigma}(2+\epsilon) [4 + 16 (\sqrt{a}r)^{2+\epsilon}]}.
\end{equation}
\end{example}

Another situation of negative potential is, instead of one potential of large size, we can also have many potentials of small size.
We prove a similar result:
\begin{thm}
  In dimension $n \geq 3$, suppose $H = -\Delta + V(x)$ is self-adjoint.
  Fix $a>0$ and $1/2 < \sigma < 4 (n-2)^2$.
  Suppose $V(x) = \sum_{j = 1}^{m} V_j(x)$, with $V_j(x)$ centered at $x= c_j$.
 If for each $j$, there is $0<\lambda_j<1$, such that $V_j(x)$ satisfies the following condition:
\begin{equation}
\lambda_j ((n-2)^2 - \sigma/4) \frac{1}{r^2 \langle x - c_j \rangle^{2\sigma}}
-2 \frac{M_{\sigma}}{\sqrt{a}} |\nabla V(x)| \geq 0,
\end{equation}
with $\lambda = \sum_{j=1}^{m} \lambda_j <1$, then $i[H, \sum_{j=1}^{m} \lambda_j \gamma_{c_j}]$ is a positive operator.
To be specific,
\begin{equation}
 i[H, \sum_{j=1}^{m} \lambda_j \gamma_{c_j}] \geq (1 - \lambda) ( 4 - \frac{\sigma}{(n-2)^2})
 \sum_{j=1}^{m} \langle x-c_j \rangle^{-\sigma} (-\lambda_j \Delta) \langle x-c_j \rangle^{-\sigma}.
\end{equation}
\end{thm}

\subsection{General Potential with One Positive Bump}
In this case, suppose the potential function is $V(x) = V_+(x) + \sum_{j=1}^{m} V_j(x)$.
Assume $V_+(x) = V_+(r)$ is radially decreasing $C^1$ potential centered at $x=0$,
 and $V_j(x)$ is a any potential function centered at $x = c_j$.

We can still utilize the cancellation lemma in this case. But notice that if $x \cdot \nabla V_j$ is positive,
$i[V_j(x), \gamma_{-c} + \gamma_{c}]$ will always be negative on the support of $V_j(x)$ even after the cancellation.
So it is better to choose $\gamma$'s symmetric with respect to $V_+(x)$, and keep the number of $\gamma$'s small.

We can prove the following theorem:
\begin{thm}
In dimension $n \geq 3$, suppose $H = -\Delta + V(x)$ is self-adjoint, and $V(x) = V_+(x) + \sum_{j = 1}^{m} V_j(x)$ as described above.
Fix $a>0$ and $1/2 < \sigma < 4 (n-2)^2$.
If for each $j$, there is $0<\lambda_j<1$, such that $V_j(x)$ satisfies the following condition:
\begin{equation}
\lambda_j ((n-2)^2 - \sigma/4) \frac{1}{r^2 \langle x - c_j \rangle^{2\sigma}}
-2 \frac{M_{\sigma}}{\sqrt{a}} |\nabla V_j(x)| \geq 0,
\end{equation}
with $\lambda = \sum_{j=1}^{m} \lambda_j < 1/2$, then $i[H, \sum_{j=1}^{m} \lambda_j (\gamma_{c_j} + \gamma_{-c_j})]$ is a positive operator.
To be specific,
\begin{equation}
\begin{split}
 i[H, \sum_{j=1}^{m} \lambda_j \gamma_{c_j}]
 \geq & (1 - 2\lambda) ( 4 - \frac{\sigma}{(n-2)^2}) \sum_{j=1}^{m} \langle x-c_j \rangle^{-\sigma} (-\lambda_j \Delta) \langle x-c_j \rangle^{-\sigma} \\
 & + ( 4 - \frac{\sigma}{(n-2)^2}) \sum_{j=1}^{m} \langle x+c_j \rangle^{-\sigma} (-\lambda_j \Delta) \langle x+c_j \rangle^{-\sigma}.
\end{split}
\end{equation}
\end{thm}

\begin{proof}
We choose $\gamma = \sum_{j=1}^{m} \lambda_j (\gamma_{c_j} + \gamma_{-c_j})$, then $i[V_+(x), \gamma] \geq 0$ by the cancellation lemma
since $\gamma_{c_j}$ and $\gamma_{-c_j}$ are symmetric with respect to $V_+(x)$.
So for the general potential with one positive bump,
\begin{equation}
\begin{split}
 i[H, \gamma] \geq& (4 - \frac{\sigma}{(n-2)^2}) \sum_{j=1}^{m} \langle x-c_j \rangle^{-\sigma} (-\lambda_j \Delta) \langle x-c_j \rangle^{-\sigma}\\
& + ( 4 - \frac{\sigma}{(n-2)^2}) \sum_{j=1}^{m} \langle x+c_j \rangle^{-\sigma} (-\lambda_j \Delta) \langle x+c_j \rangle^{-\sigma} \\
& - 4\sum_{i=1}^{m} \sum_{j=1}^{m} \lambda_j \frac{M_{\sigma}}{\sqrt{a}} |\nabla V_i(x)| \\
= & (4 - \frac{\sigma}{(n-2)^2}) \sum_{j=1}^{m} \langle x-c_j \rangle^{-\sigma} (- \lambda_j \Delta) \langle x-c_j \rangle^{-\sigma}
- 4\lambda \sum_{i=1}^{m} \frac{M_{\sigma}}{\sqrt{a}} |\nabla V_i(x)| \\
& + ( 4 - \frac{\sigma}{(n-2)^2}) \sum_{j=1}^{m} \langle x+c_j \rangle^{-\sigma} (-\lambda_j \Delta) \langle x+c_j \rangle^{-\sigma} \\
=& (1-2\lambda) (4 - \frac{\sigma}{(n-2)^2}) \sum_{j=1}^{m} \langle x-c_j \rangle^{-\sigma} (- \lambda_j \Delta) \langle x-c_j \rangle^{-\sigma}\\
& + ( 4 - \frac{\sigma}{(n-2)^2}) \sum_{j=1}^{m} \langle x+c_j \rangle^{-\sigma} (-\lambda_j \Delta) \langle x+c_j \rangle^{-\sigma}\\
& + 2\lambda \sum_{j=1}^{m} \{ (4 - \frac{\sigma}{(n-2)^2}) \langle x-c_j \rangle^{-\sigma} (- \lambda_j \Delta) \langle x-c_j \rangle^{-\sigma}
 - 2\frac{M_{\sigma}}{\sqrt{a}} |\nabla V_j(x)|\}.
\end{split}
\end{equation}
Then by assumption,
\begin{equation}
\begin{split}
i[H, \gamma] \geq & (1-2\lambda) (4 - \frac{\sigma}{(n-2)^2}) \sum_{j=1}^{m} \{\langle x-c_j \rangle^{-\sigma} (-\lambda_i \Delta) \langle x-c_j \rangle^{-\sigma} \} \\
& + ( 4 - \frac{\sigma}{(n-2)^2}) \sum_{j=1}^{m} \langle x+c_j \rangle^{-\sigma} (-\lambda_j \Delta) \langle x+c_j \rangle^{-\sigma}.
\end{split}
\end{equation}
\end{proof}

Notice that, compared to the previous result on general potential function, here we lost a factor of 2 on the condition of $\lambda$.
And we also did not utilize the positivity of $i[-\Delta, \gamma_{-c_j}]$.
So if the general potential functions are symmetric relative to the positive bump function, we can utilize $i[-\Delta, \gamma_{-c_j}]$ and regain the lost factor of 2 in some sense.

\begin{thm}
In dimension $n \geq 3$, suppose $H = -\Delta + V(x)$ is self-adjoint,
and $V(x) = V_+(x) + \sum_{j = 1}^{m} V_j(x) + \widetilde{V}_j(x)$, with $V_+(x) = V_+(r)$ is radially decreasing $C^2$ potential centered at $x=0$,
$V_j$ centered at $x=c_j$ and $\widetilde{V}_j$ centered at $x=-c_j$.
Fix $a>0$ and $1/2 < \sigma < 4 (n-2)^2$.
If for each $j$, there is $0<\lambda_j<1$, such that $V_j(x)$ and $\widetilde{V}_j(x)$ satisfy the following condition:
\begin{equation}
\begin{split}
((n-2)^2 - \sigma/4) \frac{\lambda_j}{r^2 \langle x - c_j \rangle^{2\sigma}} -2 \frac{M_{\sigma}}{\sqrt{a}} |\nabla V_j(x)| \geq 0, \\
((n-2)^2 - \sigma/4) \frac{\lambda_j}{r^2 \langle x + c_j \rangle^{2\sigma}} -2 \frac{M_{\sigma}}{\sqrt{a}} |\nabla \widetilde{V}_j(x)| \geq 0,
\end{split}
\end{equation}
with $\lambda = \sum_{j=1}^{m} \lambda_j < 1/2$, then $i[H, \sum_{j=1}^{m} \lambda_j (\gamma_{c_j} + \gamma_{-c_j})]$ is a positive operator.
To be specific,
\begin{equation}
\begin{split}
 i[H, \sum_{j=1}^{m} \lambda_j \gamma_{c_j}] \geq & (1 - 2\lambda) ( 4 - \frac{\sigma}{(n-2)^2}) \sum_{j=1}^{m} \langle x-c_j \rangle^{-\sigma} (-\lambda_j \Delta) \langle x-c_j \rangle^{-\sigma} \\
 & + (1 - 2\lambda) ( 4 - \frac{\sigma}{(n-2)^2}) \sum_{j=1}^{m} \langle x+c_j \rangle^{-\sigma} (-\lambda_j \Delta) \langle x+c_j \rangle^{-\sigma}.
\end{split}
\end{equation}
\end{thm}

\bibliographystyle{alpha}
\bibliography{SofXi-bib-1}
\end{document}